\documentclass[11pt]{amsart}

\usepackage{amsmath, amssymb, amsbsy, amscd, mathtools}
\usepackage{times, mathrsfs, euscript, thmtools}
\usepackage{epsf, overpic, psfrag, stmaryrd}
\usepackage{caption, subcaption}
\usepackage[all]{xy}
\usepackage{hyperref}
\usepackage[dvipsnames]{xcolor}
\usepackage{enumitem, bbm}
\usepackage{tikz, tikz-cd}
\usetikzlibrary{decorations.markings}
\tikzcdset{scale cd/.style={
    every label/.append style={scale=#1},
    cells={nodes={scale=#1}}}}

\usepackage{array}
\usepackage{todonotes}
\usepackage{tikz}

\usepackage{caption}
\usepackage{subcaption}
\newtheorem{theorem}{Theorem}[section]
\newtheorem{claim}[theorem]{Claim}

\newtheorem{lemma}[theorem]{Lemma}

\theoremstyle{definition}
\newtheorem{definition}[theorem]{Definition}
\newtheorem{example}[theorem]{Example}

\newtheorem{remark}[theorem]{Remark}
\newtheorem{assumption}[theorem]{Assumption}

\newtheorem{construction}[theorem]{Construction}

\numberwithin{equation}{subsection}

\DeclareMathAlphabet{\mathpgoth}{OT1}{pgoth}{m}{n}

\DeclareMathAlphabet{\mathpzc}{OT1}{pzc}{m}{it}
\newcommand{\ZZ}{\mathbb{Z}}

\newcommand{\be}{\begin{enumerate}}
\newcommand{\ee}{\end{enumerate}}
\newcommand{\op}{\operatorname}

\newcommand{\R}{\mathbb{R}}

\DeclareMathOperator{\ind}{Ind}

\DeclareMathOperator{\crit}{Crit}
\DeclareMathOperator{\Hess}{Hess}
\DeclareMathOperator{\stab}{stab}

\newcommand{\p}{\partial}

\newcommand{\des}{\mathscr{D}}
\newcommand{\asc}{\mathscr{A}}

\newcommand{\g}{\mathfrak{g}}

\usepackage{tikz}
\usepackage{pgfplots}
\usepgfplotslibrary{groupplots}
\pgfplotsset{compat=1.18}

\usetikzlibrary{calc}
\usetikzlibrary{angles, decorations.pathreplacing}
\usetikzlibrary{decorations.markings}

\tikzset{->-/.style={decoration={
			markings,
			mark=at position #1 with {\arrow{>}}},postaction={decorate}}}

\tikzset{-<-/.style={decoration={
			markings,
			mark=at position #1 with {\arrow{<}}},postaction={decorate}}}

\makeatletter
\newcommand{\labitem}[2]{%
  \def\@itemlabel{#1}
  \item
  \def\@currentlabel{#1}\label{#2}}
\makeatother

\begin{document}

\title[Equivariant Morse-Bott cohomology through stabilization]{Equivariant Morse-Bott cohomology through stabilization}

\author{Erkao Bao}
\address{School of Mathematics, University of Minnesota, Minneapolis, MN 55455}
\email{bao@umn.edu}
\urladdr{https://erkaobao.github.io/math/}

\author{Robi Huq}
\address{School of Mathematics, University of Minnesota, Minneapolis, MN 55455}

\author{Shengzhen Ning}
\address{School of Mathematics, University of Minnesota, Minneapolis, MN 55455}
\email{ning0040@umn.edu}
\urladdr{https://sites.google.com/view/shengzhenning}

\begin{abstract}
For closed manifolds with compact Lie group actions, we study Austin-Braam's Morse-theoretic construction of Borel equivariant cohomology using the technique of \emph{stabilization}. We show that a $C^1$-small equivariant perturbation produces stable invariant Morse-Bott functions. This allows us to realize the equivariant transversality and orientability assumptions in Austin–Braam’s framework by choosing generic invariant Riemannian metrics.
\end{abstract}

\maketitle

\setcounter{tocdepth}{2}
% \tableofcontents \todo{remove the table of contents before submission}
\section{Introduction}

Equivariant Morse theory has been studied extensively since the pioneering work of Atiyah-Bott \cite{atiyah1984moment} on the Yang-Mills functional over Riemann surfaces.  Equivariant Morse cohomology for general compact Lie group actions was first developed by Austin-Braam \cite{AustinBraam}.  However, their construction relies on several assumptions (see Assumption~\ref{assump:AB}) that do not hold in general. We now explain the main analytical difficulties underlying these assumptions.

When the group $G$ is not finite, $G$-invariant functions may have non-discrete critical sets and one must work with Morse-Bott functions instead of Morse functions. In the Morse-Bott setting, a typical chain model, used by 
\cite{AustinBraam}, associates to each critical submanifold its space of 
differential forms, and defines the boundary operator via pullback and 
integration along the fiber through the endpoint evaluation maps on the moduli 
spaces of gradient flow lines. From this perspective, the analytic foundation of 
the theory relies on establishing suitable transversality properties for relevant moduli spaces.

The natural transversality condition in this context is the 
\emph{Morse-Bott-Smale} condition, abbreviated \emph{MBS} (Definition \ref{def:MBS}). There is also a weaker 
notion, referred to as \emph{weakly Morse-Bott-Smale}, or \emph{weakly MBS} (Definition \ref{def:weakMBS}). The first 
set of difficulties can be summarized as follows:

\begin{enumerate}
    \item[(C1)] \label{c1}
    While the weakly MBS condition can be achieved in the non-equivariant Morse-Bott setting, it generally fails in the equivariant Morse-Bott setting.  Moreover, the weakly MBS condition alone is insufficient to define a Morse-theoretic chain complex.
    
    \item[(C2)]
    The MBS condition cannot, in general, be achieved even in the non-equivariant Morse-Bott case.
\end{enumerate}

A classical example in \cite{Latschev00}, the \emph{torus with legs} (Example~\ref{example:hornedtorus}), exhibits a Morse-Bott function for which the MBS condition fails for all metrics, although the weakly MBS condition can still be arranged.  This example justifies (C2).

Later, we introduce another example of a mapping torus (Example~\ref{example:mapping torus}), which illustrates the obstruction described in (C1).

\medskip

A further complication in Morse-Bott theory is the following:
\begin{enumerate}
    \item[(C3)]
    The moduli spaces of gradient flow lines are not always orientable.
\end{enumerate}
This usually necessitates the use of local coefficient systems or orientation bundle valued differential forms.

Taken together, these issues contribute to the relative underdevelopment of Morse-theoretic equivariant (co)homology for compact Lie group actions.

We now analyze these difficulties in greater detail. The example of torus with legs (Example~\ref{example:hornedtorus}) can be upgraded into the equivariant setting by letting $\ZZ_n$ act on it by rotations. It shows that the obstructions above persist even when the group action itself is well behaved.  In particular, the $\ZZ_n$-action on torus with legs is free, yet the function still fails to admit a Morse-Bott-Smale metric. This observation motivates a refinement of the notion of $G$-invariant Morse-Bott functions.  Rather than merely requiring $G$-invariance, we consider \emph{$G$-Morse-Bott functions} (Definition \ref{definition:G-Morse-Bott}), for which each critical submanifold is required to be a single $G$-orbit.  Restricting attention to $G$-Morse-Bott functions excludes the example of $\mathbb{Z}_n$-equivariant torus with legs, but does not resolve the obstruction in (C1).

Note that the function in the mapping torus example (Example~\ref{example:mapping torus}) is $G$-Morse-Bott.  To address this remaining difficulty, we adapt the \emph{stabilization} approach developed in \cite{mayer1989Ginvariant,bao2024morse}.  Starting from a $G$-Morse-Bott function, we perform a carefully chosen, $C^1$-small equivariant perturbation near the critical submanifolds.  The resulting function, which we call \emph{stable}, retains the essential equivariant features of the original function while exhibiting substantially improved analytical properties.  It is worth emphasizing that, in the example of torus with legs, the Morse-Bott function is stable but not $G$-Morse-Bott. The combination of stability and the $G$-Morse-Bott condition turns out to be precisely what is needed.

\medskip

\noindent
\textbf{Key observation.}
The central theme of this paper is that \emph{stable $G$-Morse-Bott functions} provide the appropriate framework of equivariant Morse theory for compact Lie group actions. By considering such class of functions, we resolve the transversality obstructions by achieving Morse-Bott-Smale condition for generic $G$-invariant metrics, and overcome orientability issues without appealing to local coefficient systems or orientation bundle valued differential forms.  As a consequence, we obtain a well-defined Morse-theoretic equivariant cochain complex that recovers Borel equivariant cohomology.

The stabilization procedure is explicit in certain cases, making concrete computations accessible.  In a subsequent paper, we will exploit this feature further by studying the limiting behavior of the equivariant Morse complex as the size of the perturbation tends to zero, thereby relating the stabilized and unstabilized theories in a precise way.

\subsection{Main results}
Let $f\in C^\infty(M)$ be a smooth function on a closed smooth manifold $M$. Let $p\in M$ be some critical point of $f$. The Hessian of $f$ at the critical point $p$ is a symmetric bilinear form Hess$_f(p):T_pM\times T_pM\rightarrow \mathbb{R}$ defined by \[\text{Hess}_f(p)(v,w)=V(Wf),\]where $V,W$ are any vector fields extending $v,w\in T_pM$. Let \[\text{ker}(\text{Hess}_f(p)):=\{v\in T_pM\,|\,\text{Hess}_f(p)(v,w)=0,\forall w\in T_pM\}.\]
\begin{definition}[Morse-Bott functions]
    A smooth function $f:M\to \mathbb{R}$ is called \emph{Morse-Bott} if its critical set $\crit(f)$ is a disjoint union of connected submanifolds such that for each connected component $S$ of $\crit(f)$ and any $p\in S$, we have $T_pS=\text{ker}(\text{Hess}_f(p))$. We denote by $\ind_f(S)$ (or simply $\ind (S)$ when the function is clear in context) the Morse index of $S$, which is the maximal dimension of a subspace in $T_pM$ where $\text{Hess}_f(p)$ is negative definite, for any $p\in S$.
\end{definition}

Now, suppose there is a smooth action of the compact Lie group \(G\) on \(M\). Let $C^{\infty}(M)^G$ be the space of $G$-invariant smooth functions on $M$. When $G$ is a Lie group, $C^{\infty}(M)^G$ may not contain any Morse function since if $p\in M$ is a critical point for some $f\in C^{\infty}(M)^G$, then all points in the orbit submanifold $G\cdot p$ are also critical points of $f$.  As a result, to study the $G$-equivariant Morse theory on $M$, it is natural to allow for critical submanifolds.

\begin{definition}[$G$-Morse-Bott functions]\label{definition:G-Morse-Bott}
A $G$-invariant function $f\in C^{\infty}(M)^G$ is called \emph{$G$-Morse-Bott} if 
\begin{itemize}
    \item $f$ is Morse-Bott;
    \item each connected component $S$ of the critical submanifolds of $f$ satisfies $T_pS=T_p(G\cdot p)$ for any $p\in S$.
\end{itemize}
\end{definition}

This notion was simply called \emph{Morse} in \cite{wasserman1969equivariant} or \emph{$G$-Morse} in \cite{baizhang20}. When $G$ is a finite group, it also coincides with the notion of equivariant Morse function introduced in \cite[Definition 1.1]{bao2024morse}. If $G$ acts trivially on $M$, then a $G$-Morse-Bott function reduces to an ordinary (non-equivariant) 
Morse function. It is important to distinguish between a $G$-invariant Morse-Bott function and a 
$G$-Morse-Bott function. For instance, constant functions are always 
$G$-invariant Morse-Bott, but they are not $G$-Morse-Bott in the sense of 
Definition\ref{definition:G-Morse-Bott} unless $G$ acts transitively on $M$. Moreover, the height function on a torus with legs in Example \ref{example:hornedtorus} can be viewed as a $\mathbb{Z}_n$-invariant Morse-Bott function, but not a $\mathbb{Z}_n$-Morse-Bott function. 

A fundamental fact, proved in \cite[Lemma 4.8]{wasserman1969equivariant}, asserts the existence result for $G$-Morse-Bott functions.

\begin{theorem}[\cite{wasserman1969equivariant}]\label{thm:wasserman}
 $G$-Morse-Bott functions are dense in $C^{\infty}(M)^G$.
\end{theorem}

Now, we consider the critical submanifolds of some $G$-invariant Morse-Bott function $f$. An orbit $G\cdot p$ consisting of critical points will be called a \emph{critical orbit}, which is diffeomorphic to the Lie group quotient $G/H$ where $H$ is the stabilizer group at $p$. Since Hess$_f(p)(v,\textrm{-})=0$ for any $v\in T_p(G\cdot p)$, by abuse of notation, Hess$_f(p)$ will also be used to denote the symmetric bilinear form on the normal space \[N_p:=T_pM/T_p(G\cdot p).\]  Thus, when $f$ is further assumed to be $G$-Morse-Bott, Hess$_f(p)$ will be non-degenerate on $N_p$. There is also a natural slice representation of $H$ on $N_p$. It admits a canonical decomposition \[N_p=N_p'\oplus N_p''\] into the trivial part $N_p'$ consisting vectors fixed by $H$, and the non-trivial part $N_p''$ consisting vectors $v\in N_p$ such that \[\int_H h\cdot v \,d\mu_h=0\]for the Haar measure $\mu$ on $H$. The slice theorem states that there is a $G$-invariant neighborhood $\mathcal{N}_p$ of $G\cdot p$ which is $G$-equivariantly diffeomorphic to $G\times_{H}N_p$ that takes $p$ to $[e,0]$. Accordingly, we also have a canonical splitting of this normal bundle \[\mathcal{N}_p=\mathcal{N}_p'\oplus \mathcal{N}_p''\] where $\mathcal{N}_p'$ is identified with the trivial bundle $G/H\times N_p'$ and $\mathcal{N}_p''$ is identified with the possibly non-trivial bundle $G\times _H N_p''$. We formulate a notion of stability for $G$-invariant Morse-Bott functions that directly generalizes the stability introduced in \cite{bao2024morse} for finite group actions.

\begin{definition}[Stable]\label{definition:stable G-Morse-Bott}
  Let $f$ be a $G$-invariant Morse-Bott function. A connected component $S$ of $\crit(f)$ is called \emph{stable} if the Hessian bilinear form Hess$_f(p)$ is positive definite on $N_p''$ for some $p\in S$. $f$ is said to be \emph{stable} if all components of $\crit(f)$ are stable. 
\end{definition}

\begin{remark}\label{rmk:stablewelldf}
    It is easy to check that if Hess$_f(p)$ is positive definite on $N_p''$, then Hess$_f(q)$ is also positive definite on $N_q''$ for any $q\in G\cdot p$. This guarantees the above is well-defined.
\end{remark}

We refer the reader to Figure \ref{fig:notions} for an illustration of the inclusion relations among the notions introduced above.
Analogous to Theorem \ref{thm:wasserman}, the study of equivariant Morse theory of stable $G$-Morse-Bott functions requires the existence of them, which is provided by the following result.

\begin{theorem}\label{thm:stabilization}
    Given any $G$-invariant Morse-Bott function $f$, one can perturb it into a stable $G$-Morse-Bott function $f'$ that is arbitrarily $C^1$-close to $f$.
\end{theorem} 

  In \cite{mayer1989Ginvariant}, Mayer introduced the so-called \emph{spezielle invariante Morse-funktion}, requiring the Morse index along any critical orbit remain  invariant under restriction to its orbit-type stratum and obtained an existence result. It is straightforward to see that the notion of spezielle invariante Morse-funktion is equivalent to our notion of \emph{stable $G$-Morse-Bott function}. In Section~\ref{section:stabilization}, we will give a proof of Theorem \ref{thm:stabilization} using the induction strategy developed in \cite{bao2024morse} for finite group actions. This recovers the existence result shown in \cite{mayer1989Ginvariant} while providing a more explicit construction. See also Remark \ref{remark:comparisonMayer}.

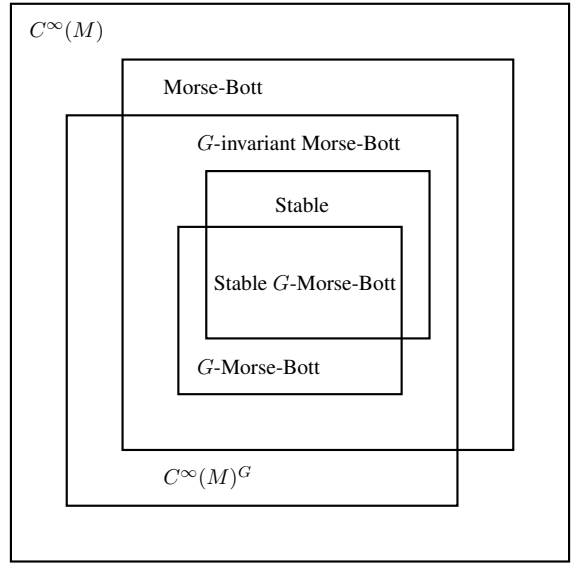
\begin{figure}
    \centering
      \resizebox{0.6\linewidth}{!}{
    \begin{tikzpicture}[x=1cm,y=1cm]

% ---------- style ----------

\tikzset{
  box/.style={
    draw=black!70!black,
    line width=1.1pt
  }
}

\draw[box] (0,0) rectangle (10,10);

\draw[box] (1,1) rectangle (8,8);

\draw[box] (2,2) rectangle (9,9);

\draw[box] (3,3) rectangle (7,6);

\draw[box] (3.5,4) rectangle (7.5,7);

\node[anchor=west] at (0.2,9.5) {$C^{\infty}(M)$};
\node[anchor=west] at (2.6,8.5)  {Morse-Bott};
\node[anchor=west] at (4.6,6.4)  {Stable};
\node[anchor=west] at (2.6,1.5)  {$C^{\infty}(M)^G$};
\node[anchor=west] at (3.2,7.5)  {$G$-invariant Morse-Bott};
\node[anchor=west] at (3.2,3.5)  {$G$-Morse-Bott};
\node[anchor=west] at (3.5,5)  {Stable $G$-Morse-Bott};

\end{tikzpicture}
}
    \caption{Various notions discussed above.}
    \label{fig:notions}
\end{figure}

Our next goal is to find a nice $G$-invariant Riemannian metric for a given stable $G$-Morse-Bott function. Given a pair $(f,g)$ of a Morse-Bott function $f$ and a Riemannian metric $g$, let $\varphi_t$ be its negative gradient flow. For any critical point $p$ or critical submanifold $S$ of $f$, we denote by
\begin{align*}
    \des_{f,g}(p)=\{x\in M\,|\,\lim_{t\rightarrow-\infty}\varphi_t(x)=p\},\asc_{f,g}(p)=\{x\in M\,|\,\lim_{t\rightarrow\infty}\varphi_t(x)=p\}\\
     \des_{f,g}(S)=\{x\in M\,|\,\lim_{t\rightarrow-\infty}\varphi_t(x)\in S\}, \asc_{f,g}(S)=\{x\in M\,|\,\lim_{t\rightarrow\infty}\varphi_t(x)\in S\}
\end{align*}
the descending/ascending manifolds at either $p$ or $S$. When the pair $(f,g)$, or the function $f$, is clear from the context, we will simply write $\asc$ and $\des$, or $\asc_g$ and $\des_g$, in place of $\asc_{f,g}$ and $\des_{f,g}$, respectively.

\begin{definition}[Morse-Bott-Smale]\label{def:MBS}
    The pair $(f,g)$ of a Morse-Bott function $f$ and a Riemannian metric $g$ is called \emph{Morse-Bott-Smale} (or  \emph{MBS} for short) if for any two components $S_1,S_2$ of $\crit(f)$, the descending manifold $\des(p)$ of any point $p\in S_1$ intersects the ascending manifold $\asc(S_2)$ of $S_2$ transversely.
\end{definition}

\begin{theorem}\label{thm:Morse-Bott-Smale-is-generic}
    Let $k \geq 1$. Let $f$ be a stable $G$-Morse-Bott function on $M$ of class $C^{k+1}$. 
    For a generic $G$-invariant Riemannian metric $g$ on $M$ of class $C^k$, the pair $(f,g)$ is Morse-Bott-Smale.
\end{theorem}

Austin-Braam \cite{AustinBraam} constructed an equivariant Morse cochain complex by using Cartan's model of $G$-invariant symmetric powers of the dual Lie algebra valued differential forms on critical submanifolds (see also \cite{zhuang2025analytictopologicalrealizationsinvariant} for the non-equivariant cochain complex). However, their construction requires Assumptions \ref{assump:AB} which do not hold in general. Therefore, Theorem~\ref{thm:Morse-Bott-Smale-is-generic} combined with Lemma~\ref{lem:orientable} shows that these assumptions can be satisfied by restricting to stable $G$-Morse-Bott functions.

The paper is organized as follows.
In Section \ref{section:stabilization}, we explain how to stabilize a $G$-Morse-Bott function. Next, we show the genericity of MBS metrics for any stable $G$-Morse-Bott function in Section \ref{section:genericity}. Then, Section \ref{section:cohomology} reviews the Morse-theoretic construction of ordinary cohomology and equivariant cohomology using $G$-invariant differential forms on critical submanifolds, following \cite{AustinBraam,zhuang2025analytictopologicalrealizationsinvariant}. Finally, in 
Section~\ref{section:example}, we present examples illustrating computations via 
the stabilization procedure.

\medskip

\textbf{Acknowledgments}
We would like to thank Tyler Lawson, 
Conan Leung, Ko Honda, and Hao Zhuang for helpful communications. This project is partially supported by the NSF grant DMS-2404529.

\section{Stabilization}\label{section:stabilization}
In this section, we adapt the proof in \cite{bao2024morse} to prove Theorem \ref{thm:stabilization}. Let $G$ be a compact Lie group acting on a closed manifold $M$. Recall that our original Definition \ref{definition:stable G-Morse-Bott} for stable $G$-invariant Morse-Bott functions is metric-independent. If, in addition, $M$ is equipped with a $G$-invariant Riemannian metric $g$, then 
for any critical orbit $G\cdot p$ the normal space $N_p$ can be canonically 
identified with the orthogonal complement $T_p(G\cdot p)^{\perp}$ with respect to 
$g$. We denote by $g_p$ the induced inner product on $N_p$. With this 
identification, the Hessian $\Hess_f(p)$ may be viewed as a self-adjoint operator 
$A_p$ on $N_p$, defined by
\[
\Hess_f(p)(v,w) = g_p(A_p v, w).
\]
In this case, $N_p$ admits a non-canonical orthogonal decomposition \[
N_{p} = N^{+}_{p} \oplus N^{-}_{p}\oplus N_p^0,
\]
where the subspaces $N^{+}_{p}$, $N^{-}_{p}$ and $N^0_p$, depending on $g_p$, denote the positive, negative and zero eigenspaces of $A_p$, respectively. Note that when $f$ is $G$-Morse-Bott, the zero eigenspace $N_p^0$ is trivial. Under the presence of $G$-invariant metric, the following lemma provides the alternative characterization of stability.

\begin{lemma}\label{lem:stabledefinition}
    Let $f$ be a $G$-invariant Morse-Bott function with some critical orbit $G\cdot p$. Then the following are equivalent.
    \begin{enumerate}
        \item $f$ is stable at $G\cdot p$.
        \item For any $G$-invariant metric $g$, $N^{-}_{p}\oplus N_p^0 \subseteq N'_{p}$. That is, $N^{-}_{p}\oplus N_p^0$ is fixed by $\stab(p)$.
        \item There exists a $G$-invariant metric $g$ such that $N^{-}_{p}\oplus N_p^0 \subseteq N'_{p}$.
        
    \end{enumerate} 
\end{lemma}
\begin{proof}
   \text{(1)$\Rightarrow$(2):} Assume $f$ is stable at $G\cdot p$ and take any $G$-invariant metric $g$. Let $H\subseteq G$ be the stabilizer group at $p$. By the $G$-invariance of $f$,\[\Hess_f(p)(v,w)=\Hess_f(p)(h\cdot v,h\cdot w)\]
    for any $h\in H$ and $v,w\in N_p$. By the definition of $A_p$,  \[g_p(A_p(h\cdot v),h\cdot w)=g_p(A_pv,w)=g_p(h\cdot A_pv,h\cdot w),\]which implies that $A_p(h\cdot v)=h\cdot A_pv$ for any $h\in H$ and $v\in N_p$. We see that $A_p(N_p')\subseteq N_p'$. Note that the canonical decomposition $N_p'$ and $N_p''$ are $g_p$-orthogonal since for any $v\in N_p',w\in N_p''$, \[g_p(v,w)=\int_Hg_p(h\cdot v,h\cdot w)d\mu_h=\int_Hg_p(v,h\cdot w)d\mu_h=g_p(v,\int_H h\cdot wd\mu_h)=0.\] As a result, we also have $A_p(N_p'')\subseteq N_p''$ since $A_p$ is self-adjoint. Let $x\in N_p^-\oplus N_p^0$ be an eigenvector with non-positive eigenvalue $\lambda$. Suppose $x=x'+x''$ in the decomposition $N_p'\oplus N_p''$. Since $A_p$ preserves $N_p'$ and $N_p''$, it follows that $A_p(x'')=\lambda x''$. By the stability assumption, $x''$ must be $0$. Therefore, we have proved $N^{-}_{p}\oplus N_p^0 \subseteq N'_{p}$.
    
    \medskip
    
    \text{(2)$\Rightarrow$(3):} By the existence of $G$-invariant metric.
    \medskip
    
    \text{(3)$\Rightarrow$(1):} Assume under some $G$-invariant metric $g$, the inclusion $N^{-}_{p}\oplus N_p^0 \subseteq N'_{p}$ holds true. We have seen in the step (1)$\Rightarrow$(2) that $N_p'$ and $N_p''$ are orthogonal. It follows that \[N_p^+=(N^{-}_{p}\oplus N_p^0 )^{\perp}\supseteq (N_p')^{\perp}=N_p'',\] which implies $N_p''$ is positive definite. Therefore, $f$ is stable at $G\cdot p$.
\end{proof}

From now on, we fix a $G$-invariant Riemannian metric $g$. Let $p\in M$ and $H\subseteq G$ be its stabilizer group. By slice theorem, there exists a $G$-invariant neighborhood $U(G\cdot p)$ of the orbit $G\cdot p$ which is $G$-equivariant diffeomorphic to $G\times_H \mathbb{D}_p$, where $\mathbb{D}_p\subseteq M$ is an $H$-invariant embedded ball of dimension $\dim M-(\dim G-\dim H)$ centered at $p$ constructed from the exponential map of some open neighborhood of the origin in the normal space $N_p=T_pM/T_p(G\cdot p)$. $\mathbb{D}_p$ will be called a local slice at $p$. There is the following basic observation.

\begin{lemma}\label{lem:correspondence}
   There exists local slice $\mathbb{D}_p$ such that the set of $G$-Morse-Bott functions on $U(G\cdot p)\cong G\times_H\mathbb{D}_p$ is in one-to-one correspondence with the set of $H$-Morse-Bott functions on $\mathbb{D}_p$.
\end{lemma}
\begin{proof}
Given a $G$-invariant function $f$ on $U(G\cdot p)$, we can consider its restriction $\tilde{f}:=f|_{\mathbb{D}_p}$ on the slice $\mathbb{D}_p$. $\tilde{f}$ is naturally $H$-invariant by the $G$-invariance of $f$. Assume $f$ is $G$-Morse-Bott. To see $\tilde{f}$ is $H$-Morse-Bott, first notice that the local slice $\mathbb{D}_p$ intersects transversely with any orbit of the group action by the $G$-equivariant diffeomorphism $U(G\cdot p)\cong G\times_H\mathbb{D}_p$. Thus, $\text{Crit}(\tilde{f})=\text{Crit}(f)\cap\mathbb{D}_p$ are submanifolds in $\mathbb{D}_p$. For any component $S\subseteq \text{Crit}(\tilde{f})$, we may assume $S=S'\cap \mathbb{D}_p$ where $S'$ is a component of $\text{Crit}(f)$. For $s\in S$, suppose it has stabilizer group $H'\subseteq H$. Since $\mathbb{D}_p\pitchfork_M(G\cdot s)$, it follows that \[\dim (H\cdot s)=\dim H-\dim H'=\dim(G\cdot s)+\dim \mathbb{D}_p-\dim M,\] which implies $T_s(H\cdot s)=T_s(G\cdot s)\cap T_s\mathbb{D}_p$. We can then verify the requirement of $H$-Morse-Bott since $T_sS=T_sS'\cap T_s\mathbb{D}_p=T_s(G\cdot s)\cap T_s\mathbb{D}_p$.

Conversely, given an $H$-Morse-Bott function $\tilde{f}$, we can define a $G$-invariant function $f([x,y]):=\tilde{f}(y)$ for any $[x,y]\in G\times_H \mathbb{D}_p$. Similar to the previous discussion, any component $S=H\cdot s\subseteq \text{Crit}(\tilde{f})$ corresponds to a component \[S':=\{[x,y]\,|\,y\in S\}\subseteq \text{Crit}(f).\] 
Since $S'$ can be identified with the homogeneous space $ G\times _H(H\cdot s)$, its tangent space will be spanned by the fundament vector field of the group action by $G$. Therefore, $f$ is $G$-Morse-Bott.

Finally, it is easy to check the above correspondence between $f$ and $\tilde{f}$ is one-to-one.
\end{proof}

Now, let $f$ be any $G$-Morse-Bott functions on $M$. We will argue by induction on the dimension of the manifold $M$ to perturb $f$ into a stable one. 
\begin{assumption}
    For $k\in \mathbb{Z}_{\geq 0}$, we say that \emph{$k$-stabilization} can be performed if Theorem \ref{thm:stabilization} holds true for any compact Lie group $G$ and closed manifold $M$ with $\dim M\leq k$.
\end{assumption}

Note that $0$-stabilization can be performed for the obvious reason. We will explain how to obtain the $k$-stabilization once $(k-1)$-stabilization can be achieved. The proof adapts the argument from the finite group case presented in \cite{bao2024morse}. Let us recall the construction in \cite[Corollary 6.4]{bao2024morse}.

\begin{construction}\label{construction:function}
    For small values $\lambda\gg \delta >0$, there exist smooth functions $\Phi_{\lambda},\Psi_{\delta}:\mathbb{R}\rightarrow\mathbb{R}$ such that
    \begin{itemize}
        \item $\Phi_{\lambda}(t)=t^2$ when $t\leq \lambda$;
        \item $\Phi_{\lambda}(t)=-t^2$ when $t\geq 3\lambda$;
        \item $\Phi_{\lambda}$ has exactly two critical points at $0,t_0$, where $\lambda<t_0<2\lambda$ and $\Phi_{\lambda}''(t_0)<0$.
        \item $\Psi_{\delta}$ is a plateau function which is constant with value $1$ inside $[t_0-\delta,t_0+\delta]$ and $0$ outside $[\lambda-\delta,3\lambda+\delta]$. 
    \end{itemize}
\end{construction}

\begin{figure}[htbp]
\centering
\begin{tikzpicture}
\begin{groupplot}[
  group style={group size=2 by 1, horizontal sep=1.2cm},
  width=0.45\textwidth,
  height=0.35\textwidth,
  grid=both,
  xlabel={$t$}
]

% --------------------------------------------------
% Left panel: Phi(t)
% --------------------------------------------------
\nextgroupplot[
  title={$\Phi_{\lambda}(t)$},
  ylabel={},
  ymin=-15, ymax=2
]

\addplot[
  domain=0:1,
  samples=100,
  smooth,
  thick
] {x^2};

% 2) Smooth interpolation on 1 <= t <= 3
% smoothstep u = (t-1)/2
% s(u) = 6u^5 - 15u^4 + 10u^3
% Phi(t) = (1 - 2 s(u)) t^2
\addplot[
  domain=1:3,
  samples=200,
  smooth,
  thick
] {(1 - 2*(6*((x-1)/2)^5 - 15*((x-1)/2)^4 + 10*((x-1)/2)^3)) * x^2};

% 3) Phi(t) = -t^2 for t >= 3
\addplot[
  domain=3:4,
  samples=100,
  smooth,
  thick
] {-x^2};
\nextgroupplot[
  title={$\Psi_{\delta}(t)$},
  ylabel={},
  ymin=-0.05, ymax=1.05
]

% Parameters
\def\tt{1.5}
\def\dd{0.15}

% rising edge
\addplot[smooth, thick, domain={1-\dd}:{\tt-\dd}, samples=100]
  { 6*((x-(1-\dd))/(\tt-1))^5
   -15*((x-(1-\dd))/(\tt-1))^4
   +10*((x-(1-\dd))/(\tt-1))^3 };

% plateau
\addplot[smooth, thick, domain={\tt-\dd}:{\tt+\dd}, samples=2] {1};

% falling edge
\addplot[smooth, thick, domain={\tt+\dd}:{3+\dd}, samples=100]
  { 1 - (
     6*((x-(\tt+\dd))/(3-\tt))^5
    -15*((x-(\tt+\dd))/(3-\tt))^4
    +10*((x-(\tt+\dd))/(3-\tt))^3 ) };

\end{groupplot}
\end{tikzpicture}
\caption{The auxiliary functions $\Phi_{\lambda}$ (left) and $\Psi_\delta$ (right) appearing
in Construction \ref{construction:function} for $\lambda=1,t_0=1.5,\delta=0.1$.}
\label{fig:PhiPsi}
\end{figure}
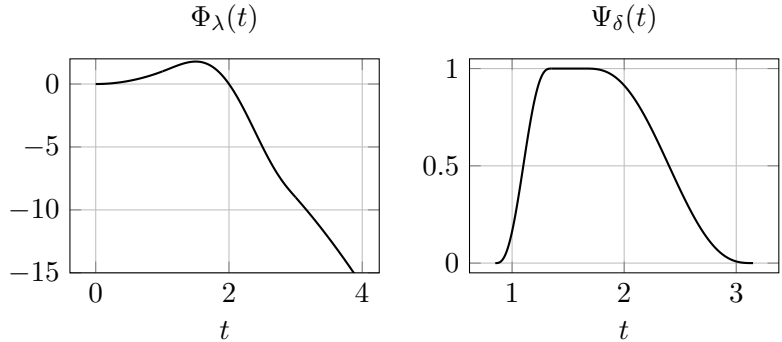

\begin{lemma}\label{lem:induction}
    Assume that $(k-1)$-stabilization can be performed. Let the compact Lie group $H$ act on $\mathbb{D}^k$, a $k$-dimensional open ball with some $H$-invariant Riemannian metric $g$. Let $f$ be an $H$-Morse-Bott function on $\mathbb{D}^k$ with a unique critical point $p\in\mathbb{D}^k$. Then, there exists another $H$-Morse-Bott function $F$ on $\mathbb{D}^k$ such that
    \begin{itemize}
        \item the critical orbits of $F$ are either $p$ or contained in the sphere centered at $p$ of radius arbitrarily smaller than the injective radius at $p$;
        \item $F$ agrees with $f$ away from a smaller ball centered at $p$ and can be arbitrarily $C^1$-close to $f$;
        \item $F$ is stable.
    \end{itemize}
\end{lemma}

\begin{proof}
By equivariant Morse lemma (\cite[Lemma 4.1]{wasserman1969equivariant}), there exists an open neighborhood $U_p$ around $p$ with coordinates $(x_1,\cdots,x_k)$ such that \begin{itemize}
    \item $U_p$ is contained in a ball with radius smaller than the injective radius at $p$;
    \item $f(x_1,\cdots,x_k)=-\sum_{i\leq l}x_i^2+\sum_{i>l}x_i^2$;
    \item $U_p^-:=\{x_{l+1}=\cdots=x_k=0\}$ and $U_p^+:=\{x_{1}=\cdots=x_l=0\}$ are $H$-invariant.
\end{itemize} Consider the sphere $S\subseteq U_p^-$ with small radius on which $H$ acts. By the assumption of $(k-1)$-stabilization, there exists a stable $H$-Morse-Bott function $h$ on $S$ since $\dim S\leq k-1$. For $\vec{x}=(x_1,\cdots,x_n)$, let $\|\vec{x}_+\|=(\sum_{i>l}x_i^2)^{\frac{1}{2}}$ and  $\|\vec{x}_-\|=(\sum_{i\leq l}x_i^2)^{\frac{1}{2}}$. Given small parameters $\lambda,\delta,\varepsilon>0$, we then define the modified function
\[F_{\lambda,\delta,\varepsilon}(x_1,\cdots,x_k)=\|\vec{x}_+\|^2+\Phi_{\lambda}(\|\vec{x}_-\|)+\varepsilon\Psi_{\delta}(\|\vec{x}_-\|)h(\frac{(x_1,\cdots,x_l)}{\|\vec{x}_-\|}).\]
Note that by taking $\lambda,\delta,\varepsilon$ to be sufficiently small, $F:=F_{\lambda,\delta,\varepsilon}$ can be arbitrarily $C^1$-close to $f$. This construction is essentially the same as \cite[Construction 6.5]{bao2024morse} by choosing $W=0$ therein\footnote{The space $W$ in \cite{bao2024morse} should be thought as the $H$-fixed part inside $U_p^-$. But the modification can actually be performed on the whole $U_p^-$.}. It follows from the computations in \cite[Proposition 6.6 (4)]{bao2024morse} that the critical submanifolds of $F_{\lambda,\delta,\varepsilon}$ are either the origin point or the critical submanifolds of $h$ inside the sphere $S$, when $\varepsilon$ is sufficiently small. Let $C\subseteq S$ be any component of the critical submanifolds. Then the computation in \cite[Proposition 6.6 (5)]{bao2024morse} shows that \[\text{Ind}_{F_{\lambda,\delta,\varepsilon}}(C)=\text{Ind}_{h}(C)+1,\] where the `$+1$' comes from the normal direction of $S\subseteq U_p^-$. This normal direction must be fixed by $H$ since $S$ is $H$-invariant. This enables us to see $F_{\lambda,\delta,\varepsilon}$ is stable from the stability of $h$. 
\end{proof}

\begin{remark}
 Let $S^{1}$ act on 
$S^{2}$ by rotation and consider the height function as an $S^{1}$-Morse--Bott 
function on $S^{2}$. Stabilizing at the critical orbit corresponding to the north 
pole produces a dimple (see Figure~\ref{fig:spherewithdimple} and Example \ref{example:sphere}). In this case, the 
function $h$ on the small circle centered at the north pole is chosen to be the 
constant function. This simple example serves as the prototype for the construction of $F_{\lambda,\delta,\varepsilon}$ in the above lemma.
\end{remark}

\begin{proof}[Proof of Theorem \ref{thm:stabilization}]
  Suppose $\dim M=k$. We will prove $k$-stabilization can be performed under the assumption that $(k-1)$-stabilization can be performed. By Theorem \ref{thm:wasserman}, we may first assume $f$ is $G$-Morse-Bott function by small perturbations. Fix some $G$-invariant Riemannian metric $g$. For any critical orbit $G\cdot p$ with stabilizer group $H\subseteq G$, we take a sufficiently small local slice $\mathbb{D}_p$ on which $H$ acts. By Lemma \ref{lem:induction}, $f|_{\mathbb{D}_p}$ admits a $C^1$-small perturbation into a stable $H$-Morse-Bott function $\tilde{F}$. By Lemma \ref{lem:correspondence}, the space of $G$-Morse-Bott functions near $G\cdot p$ corresponds to the space of $H$-Morse-Bott functions in the $H$-invariant local slice $\mathbb{D}_p$ by restrictions. Thus, $\tilde{F}$ corresponds to a $G$-Morse-Bott function $F$ on $M$ that is $C^1$-close to $f$. Since such a perturbation is local, we may assume $G\cdot p$ is the only critical orbit of $f$ and check $F$ is stable near $G\cdot p$. Let $S=G\cdot s$ be a critical orbit of $F$ and $s\in S$. By Remark \ref{rmk:stablewelldf}, we may take $s\in S\cap \mathbb{D}_p$. Assume $\text{stab}(s)=H'\subseteq H$. Note that the negative eigenspaces of the Hessian operators of $F$ and $\tilde{F}$ at $s$ can be identified since we have $\ind_F(S)=\ind_{\tilde{F}}(S\cap\mathbb{D}_p)$ by the correspondence between $F$ and $\tilde{F}$. We denote their negative eigenspace by $N_s^-$. Since $\tilde{F}$ is stable, we have $N_s^-$ is fixed by $H'$. Therefore, by Lemma \ref{lem:stabledefinition}, we see that $F$ must also be stable.
\end{proof}
\begin{remark}\label{remark:comparisonMayer}
If we perturb $f$ by retaining only the first two terms 
$\|\vec{x}_+\|^{2}+\Phi_{\lambda}(\|\vec{x}_-\|)$ in 
$F_{\lambda,\delta,\varepsilon}$, stability at $0$ is still ensured, but the 
$G$-Morse-Bott property is lost. The approach in 
\cite{mayer1989Ginvariant} is to abstractly apply Wasserman’s result 
(Theorem~\ref{thm:wasserman}) to obtain a small $G$-Morse-Bott perturbation, and 
then use a $G$-invariant cutoff function to interpolate between the two functions. 
Our method is more explicit if we start with an $G$-Morse-Bott function: we specify the function $h$ on the lower-dimensional 
sphere $S$, which provides concrete guidance for performing the stabilization in 
specific situations. See Example~\ref{example:mapping torus}.

\end{remark}

\section{Genericity of Morse-Bott-Smale metrics}\label{section:genericity}

\subsection{Morse-Bott case}
In this section, we review the case when there is no group action, i.e. $G = \{e\}$. When $G = \{e\}$, $G$-Morse-Bott functions are the same as the Morse functions. However, for comparison purposes, we consider a larger class of functions, Morse-Bott functions. 

To define Morse-Bott cohomology, depending on the choice of cochain model on the 
critical submanifolds, one may impose various transversality conditions; see, for instance,  \cite{Fukaya96,AustinBraam,hutchings2002lecture,BanyagaHurtubise,zhuang2025analytictopologicalrealizationsinvariant}. A first and typically indispensable requirement is that the moduli spaces of gradient flow lines be transversely cut out. This is the \emph{weakly Morse-Bott-Smale} condition introduced below, which is a weaker version of Definition \ref{def:MBS}.

\begin{definition}[Weakly Morse-Bott-Smale]\label{def:weakMBS}
   The pair $(f,g)$ of a Morse-Bott function $f$ and a Riemannian metric $g$ is called \emph{weakly Morse-Bott-Smale} (or \emph{weakly MBS} for short) if for any two components $S_1,S_2$ of $\crit(f)$, the descending manifold $\des(S_1) = \des_{f,g}(S_1)$ of $S_1$ intersects the ascending manifold $\asc(S_2) = \asc_{f,g}(S_2)$ of $S_2$ transversely. 
\end{definition}

\begin{lemma}[MBS implies weakly MBS]
    If the pair $(f,g)$ is Morse-Bott-Smale, then it is also weakly Morse-Bott-Smale.
\end{lemma}
\begin{proof}
    For any $x\in \des(S_1)\cap \asc(S_2)$, there will exist some $p\in S_1$ such that $x\in \des(p)\cap \asc(S_2)$. If $(f,g)$ is Morse-Bott-Smale, then we have
    \[T_xM=T_x\des(p)+T_x\asc(S_2)\subseteq T_x\des(S_1)+T_x\asc(S_2),\]
    which implies that $(f,g)$ is weakly Morse-Bott-Smale.
\end{proof}

The notion of `Morse-Bott-Smale' is strictly stronger than `weakly Morse-Bott-Smale' as indicated by the following example of Latschev \cite[Remark 2.4]{Latschev00}.
\begin{example}[Torus with legs \cite{Latschev00}]\label{example:hornedtorus}
    Consider the $2$-torus with coordinates $(\varphi,\psi)\in[0,2\pi)\times[0,2\pi)$ and the Morse-Bott function\footnote{Our function differs from that in \cite{Latschev00} by a negative sign, since we use negative gradient flows rather than gradient flows.}  $$h_n(\varphi,\psi)=-(2+\cos n\varphi)(1+\cos \psi)$$ for any positive integer $n$. It has the maximal circle \[S=\{\psi=\pi\}\] and $n$ saddle points \[\{p_k=(\frac{(2k+1)\pi}{n},0)\}_{0\leq k\leq n-1}\] as critical submanifolds of index $1$; $n$ minimal points \[\{q_k=(\frac{2k\pi}{n},0)\}_{0\leq k\leq n-1}\] as critical submanifolds of index $0$. If we embed the torus in $\mathbb{R}^3$ equipped with the standard metric and view $h_n$ as the height function as shown in Figure~\ref{fig:hornedtorus}, then the weakly Morse-Bott-Smale condition can be easily verified. However, since $h_n(S)>h_n(p_k)>h_n(q_k)$, there always exist flow trajectories from some point $s\in S$ to the saddle points $p_k$ for any metric. This implies that $\des(s)$ is not transverse to $\asc(p_k)$. Therefore, no metric could satisfy the Morse-Bott-Smale condition. 
\end{example}
\begin{figure}[htbp]
    \centering
    \begin{subfigure}[b]{0.4 \textwidth}
        \centering
        \includegraphics[width=\linewidth]{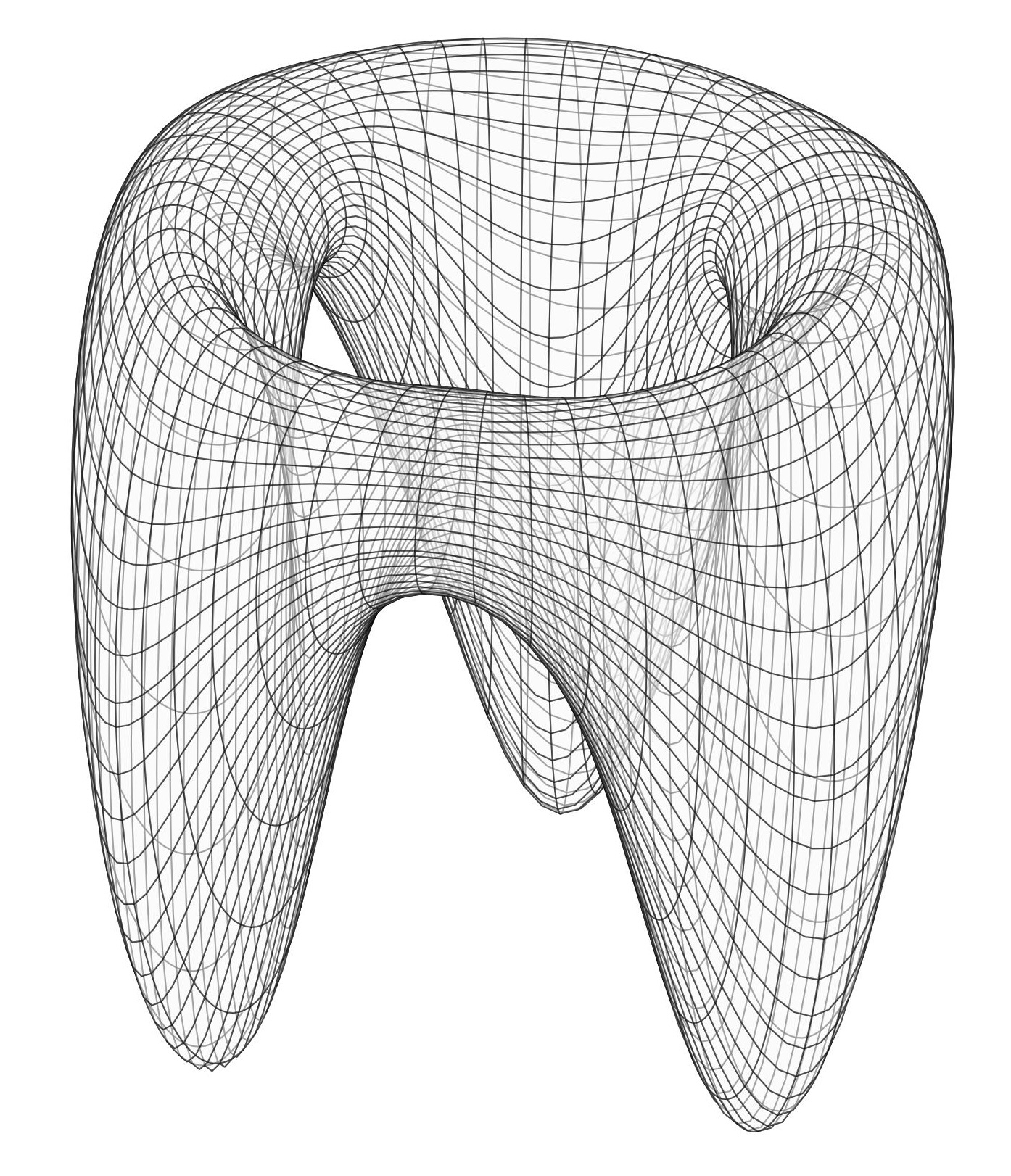}
        \caption{\small A torus with legs}
        \label{fig:hornedtorus}
    \end{subfigure}
    \hspace{2cm}
    \begin{subfigure}[b]{0.4\textwidth}
        \centering
        \includegraphics[width=\linewidth]{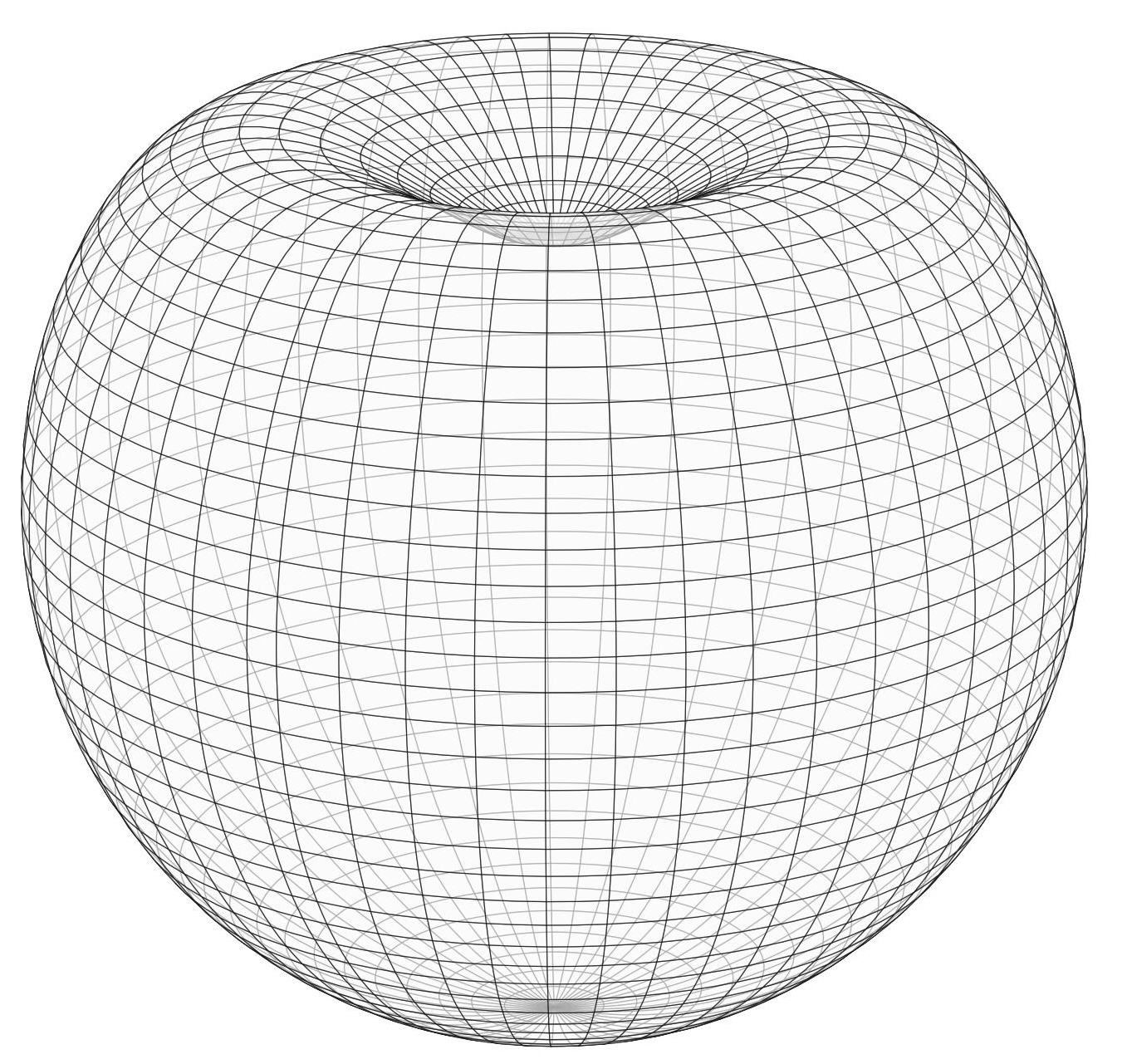}
        \caption{\small A sphere with a dimple}
        \label{fig:spherewithdimple}
    \end{subfigure}
    \label{fig:deformed_shapes}
\end{figure}

For comparison, we state the following folklore theorem (see for example \cite[ Section~6.2.1]{hutchings2002lecture} or \cite[Section 8]{ZhouMB}), whose proof is a standard modification of the Morse case (see for example \cite[Proposition~5.8]{hutchings2002lecture}).

\begin{theorem}\label{thm:weakMBS} 
    Given a Morse-Bott $C^{k+1}$-function $f$ on $M$, a generic $C^k$-metric $g$ makes the pair $(f,g)$ weakly Morse-Bott-Smale.
\end{theorem} 
% \begin{proof}
% We proceed in several steps.

%    \textbf{Step 1 (Fredholm setup):} Let us follow \cite{bourgeois2002morse} for the Fredholm setup of Morse-Bott case. We fix a metric $g_0$ on $M$ and a Whitney embedding from $M$ into  $\mathbb{R}^N$ for some large $N$ to define the Sobolev spaces. For any two components $S_-,S_+$ of $\crit(f)$, consider the Banach manifolds \[\begin{gathered}
%        \mathcal{B}(S_-,S_+):=\{u\in L_1^2(\mathbb{R},M)\,|\,\lim_{t\rightarrow\pm\infty}u(t)\in S_\pm\},\\
%        \mathcal{G}:=\{C^1 \text{-metrics on }M\}
%    \end{gathered}  \]
%    and the Banach bundle $\mathcal{E}\rightarrow\mathcal{B}(S_-,S_+)\times \mathcal{G}$ with the fiber $\mathcal{E}_{(u,g)}:=L^2(u^*TM)$. There is a section 
%    \[\mathcal{L}=\frac{d}{ds}+\text{grad}_{(f,g)}:\mathcal{B}(S_-,S_+)\times \mathcal{G}\rightarrow \mathcal{E}\]where $\text{grad}_{(f,g)}$ is the gradient vector field of $f$ with respect to the metric $g$. Let $(u,g)\in \mathcal{L}^{-1}(0)$, i.e. a gradient trajectory from $S_-$ to $S_+$. The linearized operator 
%    \[D_{(u,g)}:L_1^2(u^*TM)\times T_g\mathcal{G}\rightarrow L^2(u^*TM)\]
% is defined by composing the derivatives of $\mathcal{L}$ at $(u,g)$ and the projection onto the fiber direction $\mathcal{E}_{(u,g)}$.
%     \medskip

%     \textbf{Step 2 (Applying Sard-Smale's theorem):}

%     \medskip

%     \textbf{Step 3 (Verifying transversality):}
% \end{proof}

The previous Example \ref{example:hornedtorus} also shows Theorem \ref{thm:weakMBS} is not true if we replace `weakly Morse-Bott-Smale' by `Morse-Bott-Smale'. However, in the next section, we will show that in the presence of group symmetry, the stability condition enables us to upgrade Theorem \ref{thm:weakMBS} to `Morse-Bott-Smale', i.e., Theorem \ref{thm:Morse-Bott-Smale-is-generic}.

\subsection{\texorpdfstring{$G$-Morse-Bott case}{G-Morse-Bott case}}

In the ``problematic'' Example~\ref{example:hornedtorus}, let $G= \ZZ_n$ act on the torus by rotating the $\varphi$-coordinate by $\frac{2\pi}{n}$. Then $G$ acts freely on the torus and preserves the function $h_n$.
Hence the function $h_n$ is $G$-invariant, Morse-Bott, and stable (with respect to $G$).
But it is \emph{not} $G$-Morse-Bott, as it does not satisfy the second condition in Definition \ref{definition:G-Morse-Bott}.
The following Lemma explains the advantage of considering $G$-Morse-Bott functions.

\begin{lemma}[Weakly MBS implies MBS]\label{lem:equivMBS}
    Let the compact Lie group $G$ act on the closed manifold $M$ with a $G$-Morse-Bott function $f$ and a $G$-invariant metric $g$. Assume $G\cdot p$ and $G\cdot q$ are two critical orbits of $f$. Then $\des(p)\pitchfork_M \asc(G\cdot q)$ if and only if $\des(G\cdot p)\pitchfork_M \asc(G\cdot q)$.
\end{lemma}
\begin{proof}
    We only need to prove the `if' direction. For any $x\in \des(p)\cap\asc(G\cdot q)$, assume \[T_x\des(G\cdot p)+T_x\asc(G\cdot q)=T_xM.\] Our goal is to show \[T_x\des(p)+T_x\asc(G\cdot q)=T_xM.\] We may further assume that $x$ is arbitrarily close to the critical point $p$ due to the following simple fact.
    \begin{claim}\label{claim:transverse}
   Let $\gamma: \R \to M$ be a gradient flow line of $(f,g)$ from $p$ to some point in $G\cdot q$. Then for any $s, s' \in \R$, $\des_{g}(p)$ and $\asc_{g}(G\cdot q)$ intersect transversely in $M$ at $\gamma(s)$ if and only if they  intersect transversely at $\gamma(s')$.
\end{claim}
\begin{proof}[Proof of Claim]
This follows from the observation that the descending and ascending manifolds are invariant under the gradient flows. So, it maps $T_{\gamma(s)} \des(p)$ isomorphically to $T_{\gamma(s')} \des(p)$, and maps $T_{\gamma(s)} \asc(G\cdot q)$ isomorphically to $T_{\gamma(s')} \asc(G\cdot q)$. 
\end{proof}

    On the one hand, since the pair $(f,g)$ is $G$-invariant, the orbit $G\cdot x\subseteq\des(G\cdot p)$ which implies that 
    \[T_x\des(p)+T_x(G\cdot x)\subseteq T_x\des(G\cdot p).\]
    
    On the other hand, let us take a basis in the Lie algebra  $\g$ and consider their corresponding fundamental vector fields $X_1,\cdots,X_{\text{dim}G}$ on $M$ which span the tangent spaces of the orbits. At the point $p$, we have\[0=T_p\des(p)\cap T_p(G\cdot p)=T_p\des(p)\cap \text{span}\{X_1,\cdots,X_{\text{dim}G}\}|_p.\] For dimensional reason, it follows that \[T_p\des(G\cdot p)=T_p\des(p)+\text{span}\{X_1,\cdots,X_{\text{dim}G}\}|_p.\] Since $x$ is sufficiently close to $p$, by continuity we have \[T_x\des(G\cdot p)=T_x\des(p)+\text{span}\{X_1,\cdots,X_{\text{dim}G}\}|_x=T_x\des(p)+ T_x (G\cdot x).\] 
    
    Finally, the $G$-invariance of $(f,g)$ also implies that $T_x(G\cdot x)\subseteq T_x\asc(G\cdot q)$. Therefore, we can see that 
    \begin{align*}
        T_x\des(p)+T_x\asc(G\cdot q)&= T_x\des(p)+T_x(G\cdot x)+T_x\asc(G\cdot q)\\
        &=T_x\des(G\cdot p)+T_x\asc(G\cdot q)\\
        &=T_xM.
    \end{align*}
\end{proof}

\begin{proof}[Proof of Theorem \ref{thm:Morse-Bott-Smale-is-generic}]
    Let \(G\cdot p,G\cdot q\) be two critical orbits of the stable $G$-Morse-Bott function $f$ with \(H=\stab(p)\), \(H'=\stab(q)\).  Given any \(G\)-invariant Riemannian metric \(g\) of class $C^k$ on \(M\), by Lemma \ref{lem:equivMBS}, our goal is to show that a generic \(G\)-invariant perturbation \(g'\) of \(g\) supported near \(G\cdot p\), will achieve the weakly Morse-Bott-Smale transversality between $\des_{g'}(G\cdot p)$ and $\asc_{g'}(G\cdot q)$. 
    
    We denote by $(H)$ the set of subgroups of $G$ which are conjugate to $H$. There is a $G$-invariant submanifold (see for example \cite[Section 1.5]{Groupactionnotes})
    \begin{align*}
        M_{(H)}=\{x\in M\,|\,\text{stab}(x)\in(H)\}
    \end{align*} which has the structure of a $G/H$-bundle  \[G/H\hookrightarrow M_{(H)}\xrightarrow{\pi} X_{(H)}:=M_{(H)}/G.\] $M_{(H)}$ is the flow out $G\cdot M_H$ of the (typically non-closed) submanifold \[M_H=\{x\in M\,|\,\text{stab}(x)=H\}.\]  We first have the following observation.
\begin{claim}\label{claim:invariant}
  For any $G$-invariant metric $g'$, the gradient vector field of $(f,g')$ restricted to $M_H$ (resp. $M_{(H)}$) must be tangent to $M_H$  (resp. $M_{(H)}$). Consequently, \[\des_{f,g'}(p)\subseteq M_H, \des_{f,g'}(G\cdot p)\subseteq M_{(H)}.\]
\end{claim}
\begin{proof}[Proof of Claim]
    Since the pair $(f,g')$ is $G$-invariant, the action of $G$ preserves gradient flow lines. Thus, any gradient flow line that intersects $M_H$ can not leave $M_H$ by the uniqueness theorem of ODE. It follows that the gradient vector field restricted to $M_H$ must be tangent to $M_H$. This further implies that \[\des_{f|_{M_H},g'|_{M_H}}(p)\subseteq \des_{f,g'}(p).\] Since $f$ is stable, these two descending manifolds have the same dimension and thus must be equal. Therefore, $\des_{f,g'}(p)\subseteq M_H$. Since $M_{(H)}$ is the flow out of $M_H$, we also have the conclusion for $M_{(H)}$.
\end{proof}
    Our proof will proceed in three steps:
    \begin{enumerate}[label=\arabic*.]
       \item Upgrade transversality from $X_{(H)}$ to $M_{(H)}$ using the $G$-symmetry.
        \item Achieve transversality on the quotient space $X_{(H)}$ using Sard-Smale argument, where no group action remains.
        \item Upgrade transversality from $M_{(H)}$ to $M$ using the stability condition at the critical orbit $G\cdot q$.
    \end{enumerate}

\medskip
 \noindent\textbf{Step 1}

Restricting the pair $(f,g)$ to $M_{(H)}$, we obtain a pair 
$(\tilde f,\tilde g)$ on $M_{(H)}$. Since $f$, $g$, and $M_{(H)}$ are $G$-invariant, the pair $(\tilde f,\tilde g)$ is also $G$-invariant. Moreover, the following correspondences between $M_{(H)}$ and $X_{(H)}$ are straightforward.
\begin{itemize}
    \item There exists a function $\overline f$ on $X_{(H)}$ such that $\tilde f=\overline f\circ\pi$. Since $f$ is $G$-Morse-Bott, the function $\overline f$ is Morse with a critical point at $\pi(p)$. 
    \item The metric $\tilde{g}$ induces an orthogonal splitting of the tangent bundle \[TM_{(H)}\cong (TM_{(H)})^{\text{ver}}\oplus (TM_{(H)})^{\text{hor}}\] into the vertical part and the horizontal part \begin{gather*}
        (TM_{(H)})^{\text{ver}}:=\ker d\pi, \quad(TM_{(H)})^{\text{hor}}:=(\ker d\pi)^{\perp}.
    \end{gather*} 
     The restriction of $\tilde{g}$ to the horizontal distribution $(TM_{(H)})^{\text{hor}}$ descends to a metric $\overline{g}$ on $X_{(H)}$ such that $\pi:(M_{(H)},\tilde{g})\rightarrow (X_{(H)},\overline{g})$ is a Riemannian submersion;
    \item Under the above splitting, any metric $\overline{g}'$ on $X_{(H)}$ determines a metric $\tilde{g}'$ on $M_{(H)}$ by replacing $\tilde{g}$ on the horizontal distribution $(TM_{(H)})^{\text{hor}}\cong \pi^*TX_{(H)}$ with $\pi^*\overline{g}'$; 
    \item If $\gamma(t)$ is a gradient flow line of $(\tilde{f},\tilde{g}')$ in $M_{(H)}$, then its projection $\pi(\gamma(t))$ is a gradient flow line of $(\overline{f},\overline{g}')$ in $X_{(H)}$. 
\end{itemize}

With this understood, we now aim to show the transversality of the pair $(\overline{f},\overline{g}')$ will imply the transversality of the pair $(\tilde{f},\tilde{g}')$. Let $\phi_t$ be the negative gradient flow of $(\overline{f},\overline{g}')$ and assume $\phi_t$ exists for all time. We will use \[\des_{\overline{f},\overline{g}'}(\pi(p)):=\{x\in X_{(H)}\,|\,\lim_{t\rightarrow-\infty}\phi_t(x)=\pi(p)\}\] to denote the descending manifold at $\pi(p)$ in the usual sense. For ascending manifold, although the point $q$ may have larger stabilizer group $H'\supsetneq H$ and not belong to $M_{(H)}$, we will still use \[\asc_{\overline{f},\overline{g}'}(\pi(q)):=\{x\in X_{(H)}\,|\,\lim_{t\rightarrow+\infty}d_g(\pi^{-1}(\phi_t(x)),G\cdot q)=0\}\]
to denote the union of flow lines in $X_{(H)}$ whose lifts in $M_{(H)}\subseteq M$ converge to $G\cdot q$ in the Riemannian distance $d_g$.

\begin{claim}\label{claim:liftfromXHtoMH}
    Under the above setting, if  \[\des_{\overline{f},\overline{g}'}(\pi(p))\pitchfork_{X_{(H)}}\asc_{\overline{f},\overline{g}'}(\pi(q)),\]
    then \[\des_{\tilde{f},\tilde{g}'}(G\cdot p)\pitchfork_{M_{(H)}}\pi^{-1}(\asc_{\overline{f},\overline{g}'}(\pi(q))).\]
\end{claim}

\begin{proof}
       If $y\in \des_{\overline{f},\overline{g}'}(\pi(p))$ and $x\in \pi^{-1}(y)$, we know $\pi(\phi_t(y))$ is a negative gradient flow line in $X_{(H)}$ passing through $y$. It follows that $x\in \des_{\tilde{f},\tilde{g}'}(G\cdot p)$ since $\lim_{t\rightarrow -\infty}\pi(\phi_t(y))=\pi(p)$.
       Therefore, we see that
       \[\pi^{-1}(\des_{\overline{f},\overline{g}'}(\pi(p)))\subseteq\des_{\tilde{f},\tilde{g}'}(G\cdot p),\] which particularly implies that \[(T_xM_{(H)})^{\rm{ver}}+\pi^*T_{y}\des_{\overline{f},\overline{g}'}(\pi(p))\subseteq T_x\des_{\tilde{f},\tilde{g}'}(G\cdot p).\] Now, assume $x\in \des_{\tilde{f},\tilde{g}'}(G\cdot p)\cap\pi^{-1}(\asc_{\overline{f},\overline{g}'}(\pi(q)))$. By the transversality assumption in $X_{(H)}$,
       \begin{align*}
           T_xM_{(H)}&=(T_xM_{(H)})^{\rm{ver}}+ \pi^*T_{\pi(x)}X_{(H)}\\
           &=(T_xM_{(H)})^{\rm{ver}}+\pi^*T_{\pi(x)}\des_{\overline{f},\overline{g}'}(\pi(p))+\pi^*T_{\pi(x)}\asc_{\overline{f},\overline{g}'}(\pi(q))\\
           &\subseteq T_x\des_{\tilde{f},\tilde{g}'}(G\cdot p)+T_x\pi^{-1}(\asc_{\overline{f},\overline{g}'}(\pi(q))).
       \end{align*}
\end{proof}

\medskip
 \noindent\textbf{Step 2}
 
Now we aim to perturb the metric $\overline{g}$ near the point $\pi(p)$ on the quotient manifold $X_{(H)}$ into a generic metric $\overline{g}'$ such that the pair $(\overline{f},\overline{g}')$ satisfies the transversality assumption in Claim \ref{claim:liftfromXHtoMH}. Note that if $\overline g'$ differs from $\overline g$ only within a compact neighborhood of $\pi(p)$, then its flow $\phi_t$ will exist for all time. Indeed, by Claim~\ref{claim:invariant}, the gradient flow of $(\tilde f,\tilde g)$ on $M_{(H)}$ exists for all time. Projecting these flow lines via $\pi$ then 
implies that the gradient flow of $(\overline f,\overline g)$ on $X_{(H)}$ also exists for all time. 
 
Since there is no group action on $X_{(H)}$, we will follow the standard Sard-Smale argument (see, for example, \cite[Proposition~5.8]{hutchings2002lecture} and \cite[Theorem~3.3]{salamonMorse}) to achieve the transversality in $X_{(H)}$.

\begin{claim}\label{claim:perturbationMH}
   For a generic metric $\overline{g}'$ that coincides with $\overline{g}$ outside a compact neighborhood of $\pi(p)$, we have 
   \[\des_{\overline{f},\overline{g}'}(\pi(p))\pitchfork_{X_{(H)}}\asc_{\overline{f},\overline{g}'}(\pi(q)).\]

\end{claim}
\begin{proof}

Take the Banach manifold $\mathcal{B}\times \mathcal{G}$ where
  \begin{gather*}
       \mathcal{B}:=\{u\in W^{1,2}(\mathbb{R},X_{(H)})\,|\,\lim_{t\rightarrow-\infty}u(t)=\pi(p),\lim_{t\rightarrow+\infty}d_g(\pi^{-1}(u(t)),G\cdot q)=0\},\\
    \mathcal{G}:=\{C^k\text{-metrics on }X_{(H)}\},
   \end{gather*}
   and the Banach bundle $\mathcal{E}\rightarrow \mathcal{B}\times \mathcal{G}$ with fiber $\mathcal{E}_{(u,h)}:=L^2(u^*TX_{(H)})$. Consider the section $\mathcal{L}$ of $\mathcal{E}$ defined by \[\mathcal{L}=\frac{d}{ds}+\text{grad}_{\overline{f},h}:\mathcal{B}\times\mathcal{G}\rightarrow\mathcal{E},\]where $\text{grad}_{\overline{f},h}$ is the gradient vector field of $\overline{f}$ on $X_{(H)}$ with respect to the metric $h\in\mathcal{G}$. For any $(u,h)\in\mathcal{L}^{-1}(0)$, there is the linearized operator
   \[D_{(u,h)}:L^2(u^*TX_{(H)})\oplus T_h\mathcal{G}\rightarrow L^2(u^*TX_{(H)})\] defined by differentiating the section map $\mathcal{L}$ at $(u,h)$ followed with the  projection into the fiber direction $\mathcal{E}_{(u,h)}$.  
   
   Now, in order to apply the Sard-Smale Theorem (\cite[Theorem 5.4]{hutchings2002lecture}), we need to consider the restriction operator $D_u$ of $D_{(u,h)}$ to the $L^2(u^*TX_{(H)})$ component defined by
   \[D_{u}\xi=\nabla_{u'(t)}\xi+A_{t}\xi,\] where 
   \begin{itemize}
       \item $\nabla_{u'(t)}$ denotes the covariant derivative taken with respect to the Levi-Civita connection of a fixed background metric;
       \item $A_{t}$ is the Hessian operator of $\overline{f}$ at $u(t)$ viewed as a $(1,1)$-tensor using the metric $h$.
   \end{itemize} To see $D_u$ is a Fredholm operator, recall that by \cite[Theorem 2.1]{RSspecflow} or \cite[Principle 5.6]{hutchings2002lecture} , it suffices to show that $\lim_{t\rightarrow\pm\infty} A_t$ exist and are invertible self-adjoint operators. Similar to the usual Morse setup, the negative asymptotic limit of $A_t$ is given by the Hessian of $\overline{f}$ at the critical point $\pi(t)$. For the positive asymptotic limit, we can choose a lifting path $v:\mathbb{R}\rightarrow M_{(H)}$ such that $\pi(v(t))=u(t)$ and $v(t)$ lies in a local slice $\mathbb{D}_q$ of $G\cdot q$ at $q$ when $t\rightarrow+\infty$. 
  % Note that the closure $M_{(H)}\cap \mathbb{D}_q$ is a union of subspaces $\mathbb{D}_q^{H'}$ fixed by subgroups $H'\subseteq \text{stab}(q)$ conjugate to $H$. Without loss of generality, let us further assume that the path $v(t)$ is contained in $\mathbb{D}_q^H$ when $t\rightarrow+\infty$.
  For large $t$, if $\{e_i(t)\}$ is a frame along $u(t)$ which trivializes $u^*TX_{(H)}$, there will be a frame $\{l_i(t)\}$ tangent to $\mathbb{D}_q$ such that $\pi_*(l_i(t))=e_i(t)$. Therefore, $\lim_{t\rightarrow+\infty} A_t(e_i(t))$ can be computed by the Hessian operator of $f$ at $q$ and the limiting vectors $\lim_{t\rightarrow+\infty}l_i(t)\in T_q\mathbb{D}_q$. Since the $G$-Morse-Bott assumption implies that $\Hess_f(q)$ is non-degenerate at $T_q\mathbb{D}_q$, we see that the positive asymptotic limit of $A_t$ is also an invertible self-adjoint operator.  
  
  Therefore, we can recover the same setup as in the standard argument for obtaining Morse-Smale metrics on closed manifolds. A straightforward computation in local coordinates shows that the space of variations in $T_h\mathcal{G}$ is large enough to render $D_{(u,h)}$ surjective. Moreover, these variations can be chosen to be supported in an arbitrarily small neighborhood of $\pi(p)$ (\cite[Proposition 5.8, Step 2(a)]{hutchings2002lecture}). The Sard-Smale theorem (\cite[Theorem 5.4]{hutchings2002lecture}) then guarantees that for generic metrics $\overline{g}'\in \mathcal{G}$ on $X_{(H)}$, $D_u$ is surjective, which makes $\des_{\overline{f},\overline{g}'}(\pi(p))$ and $\asc_{\overline{f},\overline{g}'}(\pi(q))$ intersect transversely in $X_{(H)}$. 

\end{proof}
 
\medskip
 \noindent\textbf{Step 3}

By the previous two steps, we may assume that the metric $\tilde{g}'$ on $M_{(H)}$ satisfies the transversality condition $\des_{\tilde{f},\tilde{g}'}(G\cdot p)\pitchfork_{M_{(H)}}\pi^{-1}(\asc_{\overline{f},\overline{g}'}(\pi(q)))$. We now extend this perturbed metric $\tilde{g}'$ on $M_{(H)}$ to a $G$-invariant metric $g'$ on $M$ which is $C^k$-close to the original metric $g$.

Let $K\subseteq X_{(H)}$ be a neighborhood of $\pi(p)$ that contains the support of the perturbation symmetric $(0,2)$-tensor $\overline{\delta}:=\overline{g}'-\overline{g}$ in Claim \ref{claim:perturbationMH}. Then, the corresponding symmetric $(0,2)$-tensor $\tilde{\delta}:=\tilde{g}'-\tilde{g}$ on $M_{(H)}$ is supported in $\pi^{-1}(K)\subseteq M_{(H)}$. By shrinking $K$ if necessary, we may identify the normal bundle $N_{\pi^{-1}(K)}$ of the embedded submanifold $\pi^{-1}(K)\subseteq M$ with a $G$-invaiant tubular neighborhood $pr:U(\pi^{-1}(K))\rightarrow \pi^{-1}(K)$ defined by the exponential map with respect to the metric $g$. On $U(\pi^{-1}(K))$, we use the $g$-orthogonal splitting \[TU(\pi^{-1}(K))\cong pr^*(T\pi^{-1}(K))\oplus pr^*N_{\pi^{-1}(K)}\] to define a symmetric $(0,2)$-tensor $\delta$ on $M$ supported in $U(\pi^{-1}(K))$ by declaring that 
\begin{itemize}
    \item $\delta$ equals $pr^*\tilde{\delta}$ along the $pr^*(T\pi^{-1}(K))$-direction;
    \item $\delta(v,\textrm{-})=0$ if $v$ belongs to the normal direction $pr^*N_{\pi^{-1}(K)}$.  
\end{itemize}
This $\delta$ is $G$-invariant since the projection map $pr$, the splitting and $\tilde{\delta}$ are all $G$-invariant. Fix a $G$-invariant cutoff function $\chi$ supported in $U(\pi^{-1}(K))$ with $\chi\equiv 1$ on a smaller neighborhood of $\pi^{-1}(K)\subseteq U(\pi^{-1}(K))$. Now, we can define
\[g':=g+\chi \delta\] to be a $G$-invariant symmetric $(0,2)$-tensor on $M$ which satisfies that the restriction $g'|_{M_{(H)}}$ coincides with $\tilde{g}'$ on $M_{(H)}$. Note that $g'$ can be made arbitrarily $C^k$-close to $g$ if we fix the function $\chi$ and choose $\delta'$ to have sufficiently small $C^k$-norm. Since small $C^k$-perturbation preserves the positive definiteness, $g'$ defines a $G$-invariant metric on $M$ that extends the perturbation on $M_{(H)}$. 

The last ingredient for the proof of Theorem \ref{thm:Morse-Bott-Smale-is-generic} is the following claim analogous to \cite[Lemma 7.5]{bao2024morse}, which upgrades transversality in $M_{(H)}$ to transversality in $M$. This is the place where we make essential use of the stability condition at the critical orbit $G\cdot q$.

\begin{claim}\label{claim:extendtransversality}
  Let $g'$ be the metric on $M$ obtained from the metric $\tilde{g}'$ on $M_{(H)}$ by the above description. If \[\des_{\tilde{f},\tilde{g}'}(G\cdot p)\pitchfork_{M_{(H)}}\pi^{-1}(\asc_{\overline{f},\overline{g}'}(\pi(q))),\]
    then we also have \[\des_{f,g'}(G\cdot p)\pitchfork_{M}\asc_{f,g'}(G\cdot q).\]
\end{claim}
\begin{proof}
By Claim \ref{claim:invariant}, we only need to check the transversality between $\des_{\tilde{f},\tilde{g}'}(G\cdot p)$ and $\asc_{f,g'}(G\cdot q)$ at the point \[x\in \des_{f,g'}(G\cdot p)\cap\asc_{f,g'}(G\cdot q)\cap M_{(H)}=\des_{\tilde{f},\tilde{g}'}(G\cdot p)\cap\pi^{-1}(\asc_{\overline{f},\overline{g}'}(\pi(q)))\subseteq M_{(H)}.\]  
 By the transversality assumption in $M_{(H)}$, it suffices to show 
\[T_xM_{(H)}+T_x\asc_{f,g'}(G\cdot q)=T_xM.\]By Claim \ref{claim:transverse}, we can further assume that $x$ is inside a tubular neighborhood of $G\cdot q$ identified with 
\[U(G\cdot q)\cong G\times_{H'} N_q,\]where $H'$ is the stabilizer group at $q$ and $N_q$ is the normal space at $q$. Without loss of generality, let us suppose $q=[e,0]$ and $x=[e,v]\in G\times _{H'}N_q$. Then the assumption that $x\in M_{(H)}$ implies that $H'$ contains a subgroup $H''$ conjugate to $H$ such that $v$ is in the linear subspace $N_q^{H''}\subseteq N_q$ fixed by $H''$. Note that $N_q^{H''}\supseteq N_q^{H'}$. Consider the local intersection between $G\times _{H'} N_q^{H'}$ and $\asc_{f,g'}(G\cdot q)$. By the stability assumption at $G\cdot q$, we have 
\begin{align*}
    T_qM&=N_q^{H''}+T_q(\asc_{f,g'}(G\cdot q))\\
    &\supseteq N_q^{H'}+T_q(\asc_{f,g'}(G\cdot q))\\
    &=N_q^{H'}+T_q(G\cdot q)+T_q(\asc_{f,g'}(G\cdot q))\\
    &=T_{[e,0]}(G\times _{H'} N_q^{H'})+T_q(\asc_{f,g'}(G\cdot q)).
\end{align*}
Since $M_{(H)}\cap (G\times_{H'} N^{H'}_{q})$ is obtained from $G\times_{H'} N^{H'}_{q}$ by removing lower strata corresponding to larger stabilizer groups, it is a dense open submanifold of $G\times_{H'} N^{H'}_{q}$. Therefore, when $x$ is sufficiently close to $q$, we will have our desired transversality between $M_{(H)}$ and $\asc_{f,g'}(G\cdot q)$ at $x$ by continuity.

\end{proof}

Finally, note that there are only finitely many critical orbits. To complete the proof of Theorem \ref{thm:Morse-Bott-Smale-is-generic}, we just need to follow the above scheme to perform generic perturbation of the metric near each critical orbit $G\cdot m$ inside $M_{(\text{stab}(m))}$ and then extend the perturbation to $M$ in a $G$-invariant manner.

\end{proof}

    \begin{remark}
The above proof of equivariant transversality for Lie group actions differs from that of \cite{bao2024morse} for finite group actions in that we work with the stratum $M_{(H)}$ and its quotient $X_{(H)}$, rather than the submanifold of $H$-fixed points
\[
M^{H}=\{x\in M \mid H\subseteq \operatorname{stab}(x)\}.
\]
When $G$ is finite, the critical orbit $G\cdot p$ is a discrete subset of $M^{H}$, so that $p$ admits a sufficiently small neighborhood $V_{p}\subseteq M^H$ that is disjoint from the other points of $G\cdot p$. In that setting, the generic perturbation in \cite{bao2024morse} is carried out within $V_{p}$. In the Lie group setting, the perturbation is instead performed inside a small neighborhood $K$ of $\pi(p)\in X_{(H)}$.
\end{remark}

\section{Morse-theoretic equivariant cohomology}\label{section:cohomology}
\subsection{Moduli spaces of Morse flow lines and their orientability}

In \cite[Section~3.2]{AustinBraam}, Austin--Braam introduced the following assumptions in order to define a Morse-Bott cochain complex associated to a pair $(f,g)$.

\begin{assumption}\label{assump:AB}
\hfill
\begin{enumerate}
    \item The function $f$ is \emph{weakly self-indexing}. That is, for any connected components $S,S'\subseteq\crit(f)$, there are no negative gradient flow lines from $S$ to $S'$ if $\ind(S)\le \ind(S')$.
    \item For any connected components $S,S'\subseteq\crit(f)$ and any $p\in S$, the descending manifold $\des(p)$ intersects the ascending manifold $\asc(S')$ transversely.
    \item Each connected component $S\subseteq\crit(f)$, as well as its negative normal bundle $\mathcal{N}^-(S)$, is orientable.
\end{enumerate}
\end{assumption}

As pointed out by Banyaga--Hurtubise \cite[Lemma~3.6]{BanyagaHurtubise}, Condition~$(2)$ already implies Condition~$(1)$.

If we restrict attention to \emph{stable} $G$-Morse-Bott functions, then Condition~$(2)$ holds for a generic $G$-invariant metric by Theorem~\ref{thm:Morse-Bott-Smale-is-generic}.
In contrast, without the $G$-Morse-Bott and stability assumption, Condition~$(2)$ generally fails.
Indeed, if we do not restrict to $G$-Morse-Bott functions, Example~\ref{example:hornedtorus} violates Condition~$(2)$ for every $G$-invariant metric.
Even within the class of $G$-Morse-Bott functions, Example~\ref{example:mapping torus} shows that Condition~$(2)$ may still fail for all $G$-invariant metrics.

In the presence of Lie group symmetries, the critical components $S$ are homogeneous spaces, which may themselves be non-orientable (for instance $\mathbb{RP}^2 = SO(3)/O(2)$).
Nevertheless, we will see that orientability of the critical submanifolds $S$ is not required for our construction.
Moreover, the following lemma shows that the stability condition in Definition~\ref{definition:stable G-Morse-Bott} guarantees the orientability of the negative normal bundles appearing in Condition~$(3)$.

\begin{lemma}\label{lem:orientable}
Let $G$ be a compact Lie group acting on $M$.
If $f$ is a stable $G$-Morse-Bott function, then for every critical orbit $S$, the negative normal bundle $\mathcal{N}^-(S)$ with respect to any $G$-invariant Riemannian metric is a trivial bundle.
\end{lemma}

\begin{proof}
    Choose any $p\in S$ and let $H=\text{stab}(p)$. The induced representation of $H$ on the normal space $N_p$ at $p$ decomposes as the sum of the trivial part $N_p'$ and the non-trivial part $N_p''$. By Lemma \ref{lem:stabledefinition}, the stability assumption of $f$ implies that the negative eigenspace $N_p^-$ of $\text{Hess}_f(p)$ viewed as a self-adjoint operator for any $G$-invariant metric is contained in $N_p'$. This means $N_p^-$ is fixed by $H$. It follows that the negative normal bundle $\mathcal{N}^-(S)$ can be identified with \[G\times_H N_p^-\cong G/H\times N_p^-\cong S\times N_p^-,\] which is a trivial bundle over $S$.
\end{proof}

For any two connected components $S,S'\subseteq\crit(f)$, let $\mathcal{M}(S,S')$ denote the moduli space of unparametrised negative gradient flow lines from $S$ to $S'$, defined by
\[
\mathcal{M}(S,S') = (\mathscr{D}(S)\cap \mathscr{A}(S'))/\R,
\]
where $\R$ acts by translation along the flow direction.
If the pair $(f,g)$ is weakly Morse-Bott-Smale, then $\mathcal{M}(S,S')$ is a smooth manifold of dimension
\[
\ind(S)-\ind(S')+\dim S-1.
\]

Similarly, for a critical point $p\in S$ and a critical submanifold $S'\subseteq\crit(f)$, we define the moduli space $\mathcal{M}(p,S')$ of unparametrised negative gradient flow lines from $p$ to $S'$ by
\[
\mathcal{M}(p,S') = (\mathscr{D}(p)\cap \mathscr{A}(S'))/\R,
\]
where again $\R$ acts by translation in the flow direction.
When the pair $(f,g)$ is Morse-Bott-Smale, $\mathcal{M}(p,S')$ is a smooth manifold of dimension
\[
\ind(p)-\ind(S')-1.
\]

By Lemma \ref{lem:orientable}, we can fix orientations of the negative normal bundles $\mathcal{N}^-(S)$ for all connected components $S\subseteq\crit(f)$, so that the orientations are $G$-invariant.
For each $p\in S$, the identification
\[
T_p\mathscr{D}(p)\cong \mathcal{N}^-(S)_p
\]
induces an orientation on the descending manifold $\mathscr{D}(p)$ at $p$.
The moduli spaces $\mathcal{M}(p,S')$ are oriented as follows.
For any $[\gamma]\in\mathcal{M}(p,S')$, the linearized flow along $\gamma$ induces an isomorphism of vector spaces
\[
T_p\mathscr{D}(p)\cong T_{[\gamma]}\mathcal{M}(p,S') \oplus \R\langle \p_s\rangle \oplus T_q\mathscr{D}(q),
\]
where $q=\gamma(+\infty)\in S'$ and $s$ denotes the flow parameter.
We orient $T_{[\gamma]}\mathcal{M}(p,S')$ so that this isomorphism is orientation-preserving, using the chosen $G$-invariant orientations of $\mathcal{N}^-(S)$ and $\mathcal{N}^-(S')$ to orient $T_p\mathscr{D}(p)$ and $T_q\mathscr{D}(q)$ respectively. Note that the orientations assigned in this way are also $G$-invariant: for any $g\in G$, the orientation on $T_{[\gamma]}\mathcal{M}(p,S')$ coincides with the orientation on $T_{[g\cdot \gamma]}\mathcal{M}(g\cdot p,S')$.

\subsection{Cochain complex for ordinary cohomology}

We briefly recall the construction of ordinary cohomology under the Lie group symmetry using equivariant Morse-Bott theory. We refer to \cite{zhuang2025analytictopologicalrealizationsinvariant} for more details. 
Let $f$ be a stable $G$-Morse-Bott function on a closed $G$-manifold $M$, and let $g$ be a $G$-invariant Riemannian metric such that the pair $(f,g)$ satisfies the Morse-Bott-Smale condition.  
For each $i$, let $S_i$ denote the union of components of $\crit(f)$ with index $i$.

The Morse-Bott cochain complex $(C,\p)$ is defined by
\[
C^p = \bigoplus_{i+j=p} \Omega^j(S_i)^G,
\]
where $\Omega^j(S_i)^G$ denotes the space of smooth $G$-invariant differential $j$-forms on $S_i$.

For a homogeneous element $\omega \in \Omega^j(S_i)^G$, the differential is defined as a sum
\[
\p = \sum_{r \ge 0} \p_r,
\]
with components specified as follows.

\begin{itemize}
\item For $r>0$,
\[
\p_r \omega
:= (-1)^j (e_-^{\,i+r,i})_*(e_+^{\,i+r,i})^*\omega
\in \Omega^{j-r+1}(S_{i+r})^G.
\]
Here
\begin{gather*}
    e_+^{\,i+r,i} : \mathcal{M}(S_{i+r},S_i) \to  S_i,\\
   e_-^{\,i+r,i} : \mathcal{M}(S_{i+r},S_i) \to S_{i+r},
\end{gather*}
denote the endpoint evaluation maps on the moduli space of gradient flow lines.  
The operator $(e_+^{\,i+r,i})^*$ denotes the pullback, while $(e_-^{\,i+r,i})_*$ denotes pushforward, which can be interpreted as integration along the fiber since the Morse-Bott-Smale condition guarantees that $e_-^{\,i+r,i}$ has the structure of a locally trivial fiber bundle.  
The sign and degree shift reflect the dimension of the fibers, and the pushforward uses the orientations fixed on the moduli spaces in the previous subsection.

\item For $r=0$,
\[
\p_0 \omega := d\omega \in \Omega^{j+1}(S_i)^G.
\]
\end{itemize}

\begin{lemma}\cite[Proposition 2.3]{zhuang2025analytictopologicalrealizationsinvariant}
$\p^2 = 0$.
\end{lemma}

\begin{theorem}\cite[Theorem 1.4]{zhuang2025analytictopologicalrealizationsinvariant}
The cohomology of $(C,\p)$ is isomorphic to the de Rham cohomology $H^*(M;\R)$.
\end{theorem}

\begin{remark}
 In \cite{zhuang2025analytictopologicalrealizationsinvariant}, an additional 
complication in defining the cochain complex arises from considering $G$-invariant differential forms valued in 
the orientation bundle of $\des(S_i)$ over $S_i$. Under our stability assumption, this issue 
is avoided. Indeed, the locally trivial fiber bundle $
e_-^{\,i+r,i} \colon \mathcal{M}(S_{i+r},S_i) \to S_{i+r}$
is fiberwise orientable, since the orientations on $\mathcal{M}(p,S_i)$ are $G$-invariant for 
each $p\in S_{i+r}$. Consequently, integration along the fiber is well-defined 
using ordinary $G$-invariant differential forms.

\end{remark}

\subsection{Cochain complex for equivariant cohomology}
We keep the same setup and notation as in the previous subsection.
Following Austin-Braam \cite{AustinBraam}, we define the Morse-theoretic Cartan model $(C_G,\p_G)$ by
\[
C_G^p
= \bigoplus_{i+j+2k=p}
\bigl(\Omega^j(S_i)\otimes S^k(\g^*)\bigr)^G.
\]
Here $S^k(\g^*)$ denotes the $k$th symmetric power of the dual Lie algebra, and the superscript $G$ denotes the invariant subspace.  
The group $G$ acts on $\Omega^j(S_i)$ by pullback and on $S^k(\g^*)$ via the coadjoint action.

For a homogeneous element $\omega \otimes \phi \in \bigl(\Omega^j(S_i)\otimes S^k(\g^*)\bigr)^G$, the differential is defined by
\[
\p_G = \sum_{r \ge 0} (\p_G)_r,
\]
where the components are given as follows.

\begin{itemize}
\item For $r>0$,
\[
(\p_G)_r(\omega \otimes \phi)
:= (-1)^j (e_-^{\,i+r,i})_*\bigl((e_+^{\,i+r,i})^*\omega\bigr)\otimes\phi
\in \bigl(\Omega^{j-r+1}(S_{i+r})\otimes S^k(\g^*)\bigr)^G.
\]

\item For $r=0$,
\begin{align*}
(\p_G)_0(\omega\otimes\phi)
&:= d\omega\otimes\phi
- \sum_\alpha \iota_{\widetilde X_\alpha}\omega \otimes (\theta^\alpha\cdot\phi) \\
&\in \bigl(\Omega^{j+1}(S_i)\otimes S^{k}(\g^*)\bigr)^G
\oplus
\bigl(\Omega^{j-1}(S_i)\otimes S^{k+1}(\g^*)\bigr)^G,
\end{align*}
where $\{X_\alpha\}$ is a basis of $\g$, $\{\theta^\alpha\}$ is the dual basis of $\g^*$, $\widetilde X_\alpha$ denotes the fundamental vector field induced by $X_\alpha$, and $\theta^\alpha\cdot\phi$ denotes symmetric multiplication.
\end{itemize}

\begin{lemma}[\cite{AustinBraam} Theorem~5.1]
$\p_G^2 = 0$.
\end{lemma}

\begin{theorem}[\cite{AustinBraam} Theorem~5.3]
The cohomology of $(C_G,\p_G)$ is isomorphic to the Borel equivariant cohomology $H_G^*(M;\R)$ as an $S^\bullet(\g^*)^G$-module.
\end{theorem}

\section{Examples}\label{section:example}

We begin with a simple example in which stabilization is not necessarily required, but nevertheless 
yields the same computational result.

\begin{example}[$S^2$ with $S^1$-rotation]\label{example:sphere}
Let $G=S^1$. We consider the standard $S^1$-action on $S^2$ given by rotation about the $z$-axis.  
Let
\[
f(x,y,z) = z
\]
be the height function. Then $f$ is an $S^1$-Morse-Bott function with two critical submanifolds:
\[
S_N = \{(0,0,1)\}, \qquad S_S = \{(0,0,-1)\},
\]
with indices $2$ and $0$, respectively.  This example appears as \cite[Example~5.4.1 (1)]{AustinBraam}.  
Let $g$ be the standard round metric. Although the function $f$ is not stable at $S_N$, the pair $(f,g)$ already satisfies the Morse-Bott-Smale condition so that the cochain complex $(C_f^{\bullet},\partial_G)$ associated to $f$ can be defined. Let $n$, $s$, and $\theta$ denote generators in 
$\Omega^{0}(S_N)^{G}$, $\Omega^{0}(S_S)^{G}$, and $S(\mathfrak{g}^{*})^{G}$, 
respectively. Then the group $C_f^{p}$ is generated by 
$n\otimes \theta^{k}$ and $s\otimes \theta^{k}$ when $p=2k\ge 0$, and is zero 
otherwise. For degree reasons, the differential $\partial_G$ all vanishes.

On the other hand, we can modify the function $f$ by performing the stabilization at $S_N$. By pushing the north pole slightly downward, one obtains a small circle forming a new critical submanifold $S_N'$ of index $1$, while the original critical submanifold $S_N$ becomes of index $0$; see Figure~\ref{fig:spherewithdimple}. 
After stabilization, this becomes essentially equivalent to \cite[Example~5.4.1 (2)]{AustinBraam} where the $S^1$-Morse-Bott function is taken to be \[f'(x,y,z)=1-z^2.\]
We still use the notations $n,s,\theta$ as above and let $n'_0,n'_1$ be the newly created generators in $\Omega^0(S_N')^G,\Omega^1(S_N')^G$ respectively. Then the cochain group $C_{f'}^p$ associated to the stabilized function $f'$ is generated by $n_0'\otimes \theta^k$ when $p=2k+1\geq 0$; by $n\otimes \theta^k,s\otimes \theta^k, n_1'\otimes \theta^{k-1}$ when $p=2k\geq 0$; and is zero otherwise. The differential  $\partial_{G}$ is given by
\[\partial_{G}(n_0'\otimes \theta^k)=0,\,\partial_{G}(n_1'\otimes \theta^{k-1})=-n_0'\otimes\theta^k,\,\partial_{G}(n\otimes \theta^k)=\partial_{G}(s\otimes\theta^k)=n_0'\otimes\theta^k.\]

Note that the cohomology of both cochain complexes associated to $f$ and its stabilization $f'$ computes the well-known equivariant cohomology groups 
\[
H_{S^1}^n(S^2;\mathbb{R}) =
\begin{cases}
\mathbb{R}, & n=0,\\
\mathbb{R}\oplus\mathbb{R}, & n\in 2\mathbb{Z}_{>0},\\
0, & \text{ otherwise}.
\end{cases}
\]
\end{example}

The next example lies outside the scope of \cite{AustinBraam} due to the non-existence of a Morse-Bott-Smale metric, where the stabilization is really needed for obtaining the cochain complex.

\begin{example}[Mapping torus]\label{example:mapping torus}
    Consider the mapping torus
		\[
		M = T^2 \times [0,1] / {\sim},
		\]
		where $T^2 = (\R/\ZZ)^2$ and the identification is given by
		\[
		(\theta_1, \theta_2, 0) \sim (-\theta_1, \theta_2, 1).
		\]
		Note that $M$ can be alternatively described as the product between $S^1$ and Klein bottle. Let $G = S^1$ act on $M$ via the flow generated by the vector field $\p_\tau$, where $\tau$ denotes the coordinate on the $[0,1]$-factor.
		
		Define an $S^1$-invariant function $f: M \to \R$ by
		\[
		f(\theta_1, \theta_2, \tau) = (3 + \cos 2\pi \theta_1)\sin 2\pi \theta_2 .
		\]
		Then $f$ is $G$-Morse-Bott with four critical orbits:
		\begin{itemize}
			\item one index-$2$ submanifold
			\[
			P_2 = \bigl\{(0, \tfrac14, \tau) \mid \tau \in [0,1]\bigr\},
			\]
			\item two index-$1$ submanifolds
			\[
			Q_1 = \bigl\{(\tfrac12, \tfrac14, \tau) \mid \tau \in [0,1]\bigr\},
			\qquad
			R_1 = \bigl\{(\tfrac12, \tfrac34, \tau) \mid \tau \in [0,1]\bigr\},
			\]
			\item one index-$0$ submanifold
			\[
			S_0 = \bigl\{(0, \tfrac34, \tau) \mid \tau \in [0,1]\bigr\}.
			\]
		\end{itemize}
		
		The function $f$ is not stable, since $P_2$ and $R_1$ fail to be stable critical submanifolds. Moreover, the descending manifolds of $P_2$ and $R_1$ are non-orientable.
		
		In fact, there exists no $G$-invariant metric $g$ for which the pair $(f,g)$ satisfies the Morse-Bott-Smale condition. Indeed, for any $p \in Q_1$, the virtual dimension
		\[
		\op{virdim} \mathcal{M}(p, R_1) = -1,
		\qquad 
		\mathcal{M}(p, R_1) \neq \emptyset .
		\]
		To justify the latter statement, observe that $Q_1$ is stable and the stabilizer satisfies $\stab(p) \cong \ZZ_2$. Let $K$ denote the connected component of $M^{\ZZ_2}$ containing $p$. Then $K$ is diffeomorphic to a Klein bottle and contains $R_1$. Since $Q_1$ is stable, the descending manifold $\des(p)$ is contained in $K$ by Claim \ref{claim:invariant}. Consequently, the moduli space $\mathcal{M}(p, R_1)$ of flow lines from $p$ to $R_1$ in $M$ coincides with the corresponding moduli space in $K$, computed using the restricted function $f|_K$ and metric $g|_K$. The function $f|_K$ has only two critical submanifolds: the index-$1$ submanifold $Q_1$ and the index-$0$ submanifold $R_1$. Therefore, any Morse flow line starting at $p$ must terminate at $R_1$, proving that $\mathcal{M}(p, R_1)$ is nonempty.

        In fact, for any $G$-invariant metric $g$, the pair $(f,g)$ even fails to satisfy the weak Morse-Bott-Smale condition using a similar argument. We leave it to the reader to verify that.
		
		By applying stabilization, we modify $f$ to a stable $G$-Morse-Bott function $f'$. Under this process, the original critical submanifold $P_2$ gives rise to a stable critical submanifold $\bar{P}_1$ of index $1$, together with a new stable critical submanifold $P_2'$ of index $2$ located near $\bar{P}_1$ and intersecting each fiber over $\tau$ in two points. Similarly, the original critical submanifold $R_1$ becomes a stable critical submanifold $\bar{R}_0$ of index $0$, and we introduce a new stable critical submanifold $R_1'$ of index $1$ near $\bar{R}_0$, again intersecting each fiber over $\tau$ in two points. An illustration is provided in Figure~\ref{fig:deformed_shapes2}.

\medskip
\noindent\textbf{Cohomology \(H^*(M;\R)\).}
Fix a $G$-invariant product metric $g$ so that $(f',g)$ satisfies the Morse-Bott-Smale condition. For each
\[
A_i \in \{S_0,\bar R_0,R_1',Q_1,\bar P_1,P_2'\},
\]
there are generators $a_{ij}\in\Omega^j(A_i)^G$, $j=0,1$, where the decoration specifies the component, $i$ denotes the index, and $j$ denotes the degree of the differential form. The cochain complex $(C^\bullet,\p)$ has differential
\[
\begin{aligned}
\p s_{00}      &=  r'_{10}, &\qquad
\p \bar r_{00} &=  -r'_{10}, &\qquad
\p r'_{10}     &= 0, \\
\p s_{01}      &= - r'_{11},&
\p \bar r_{01} &=  r'_{11},&
\p r'_{11}     &= 0, \\
\p q_{10}      &=  p'_{20},&
\p \bar p_{10} &= - p'_{20},&
\p p'_{20}     &= 0, \\
\p q_{11}      &= - p'_{21},&
\p \bar p_{11} &=  p'_{21},&
\p p'_{21}     &= 0.
\end{aligned}
\]
We explain only how to derive the identity $\partial s_{01}=-r_{11}'$; the 
remaining computations are analogous. There are three relevant moduli spaces to 
consider:
\begin{itemize}
  \item The moduli space $\mathcal{M}(R_1',S_0)$ is diffeomorphic to $R_1'$ and 
  admits a double cover of $S_0$ via the endpoint map. Its contribution 
  to the differential yields the term $-r_{11}'$.
  
  \item The moduli space $\mathcal{M}(\bar P_1,S_0)$ is diffeomorphic to 
  $\bar P_1 \bigsqcup \bar P_1 \cong S_0 \bigsqcup S_0$ and likewise admits a double 
  cover of $S_0$ through the endpoint map. However, the two points in each fiber 
  carry opposite orientations, and hence their contributions to the differential 
  cancel.
  
  \item The moduli space $\mathcal{M}(P_2',S_0)$ does not contribute to the 
  differential. Indeed, the $1$-form on $S_0$ representing $s_{01}$ lies in the 
  $d\tau$-direction, while integration along the fiber, whose tangent space is 
  tangent to the torus fiber direction, annihilates it.
\end{itemize}
 The resulting cohomology of $(C^\bullet,\partial)$ computes the ordinary cohomology
\[
H^n(M;\mathbb{R}) =
\begin{cases}
\mathbb{R}, & n=0,\\
\mathbb{R}^2, & n=1,\\
\mathbb{R}, & n=2,\\
0, & \text{otherwise},
\end{cases}
\]
%\[
%H^0(M;\R)\cong\R,\quad
%H^1(M;\R)\cong\R^2,\quad
%H^2(M;\R)\cong\R,\quad
%H^k(M;\R)=0 \quad (k\ge3).
%\]
where $H^0$ is generated by $s_{00}+\bar r_{00}$; $H^1$ is generated by 
$s_{01}+\bar r_{01}$ and $q_{10}+\bar p_{10}$; and $H^2$ is generated by $q_{11}+\bar p_{11}$.

\medskip
\noindent\textbf{Equivariant cohomology $H_{S^1}^*(M;\R)$.}
Let $\theta$ denote a generator of $S(\g^*) \cong \R[\theta]$. The equivariant cochain complex $C_G^\bullet$ is generated by
\begin{align*}
& s_{00} \otimes \theta^k,\;
\bar{r}_{00} \otimes \theta^k,\;
r_{10}' \otimes \theta^k,\;
q_{10} \otimes \theta^k,\;
\bar{p}_{10} \otimes \theta^k,\;
p_{20}' \otimes \theta^k, \\
& s_{01} \otimes \theta^k,\;
\bar{r}_{01} \otimes \theta^k,\;
r_{11}' \otimes \theta^k,\;
q_{11} \otimes \theta^k,\;
\bar{p}_{11} \otimes \theta^k,\;
p_{21}' \otimes \theta^k,
\end{align*}
for all $k \ge 0$.

The differential $\p_G$ is given by
\begin{align*}
& \p_G(s_{00} \otimes \theta^k)       =  r_{10}' \otimes \theta^k, &
& \p_G(s_{01} \otimes \theta^k)       = - r_{11}' \otimes \theta^k - s_{00} \otimes \theta^{k+1}, \\[0.5ex]
& \p_G(\bar{r}_{00} \otimes \theta^k) =  -r_{10}' \otimes \theta^k, &
& \p_G(\bar{r}_{01} \otimes \theta^k) =  r_{11}' \otimes \theta^k - \bar{r}_{00} \otimes \theta^{k+1}, \\[0.5ex]
& \p_G(r_{10}' \otimes \theta^k)  = 0, &
& \p_G(r_{11}' \otimes \theta^k)  = - r_{10}' \otimes \theta^{k+1}, \\[0.5ex]
& \p_G(q_{10} \otimes \theta^k)   =  p_{20}' \otimes \theta^k, &
& \p_G(q_{11} \otimes \theta^k)   =  -p_{21}' \otimes \theta^k - q_{10} \otimes \theta^{k+1}, \\[0.5ex]
& \p_G(\bar{p}_{10} \otimes \theta^k)  =  -p_{20}' \otimes \theta^k, &
& \p_G(\bar{p}_{11} \otimes \theta^k)  =  p_{21}' \otimes \theta^k - \bar{p}_{10} \otimes \theta^{k+1}, \\[0.5ex]
& \p_G(p_{20}' \otimes \theta^k)      = 0, &
& \p_G(p_{21}' \otimes \theta^k)      = - p_{20}' \otimes \theta^{k+1}.
\end{align*}

The cohomology of $(C_G^\bullet,\p_G)$ is generated by $s_{00}+\bar{r}_{00}$ in degree $0$ and by $q_{10}+\bar{p}_{10}$ in degree $1$. Consequently, it computes the $S^1$-equivariant cohomology
\[
H_{S^1}^n(M;\mathbb{R}) =
\begin{cases}
\mathbb{R}, & n=0,1,\\
0, & \text{ otherwise}.
\end{cases}
\]

%\[
%H_G^k(M;\R) \cong \R \quad \text{for } k = 0,1,
%\qquad
%H_G^k(M;\R) = 0 \quad \text{otherwise}.
%\]

\begin{figure}[htbp]
    \centering
    \begin{subfigure}[b]{0.4\textwidth}
        \centering
        \begin{tikzpicture}
            \node[anchor=south west, inner sep=0] (img)
                {\includegraphics[width=\linewidth]{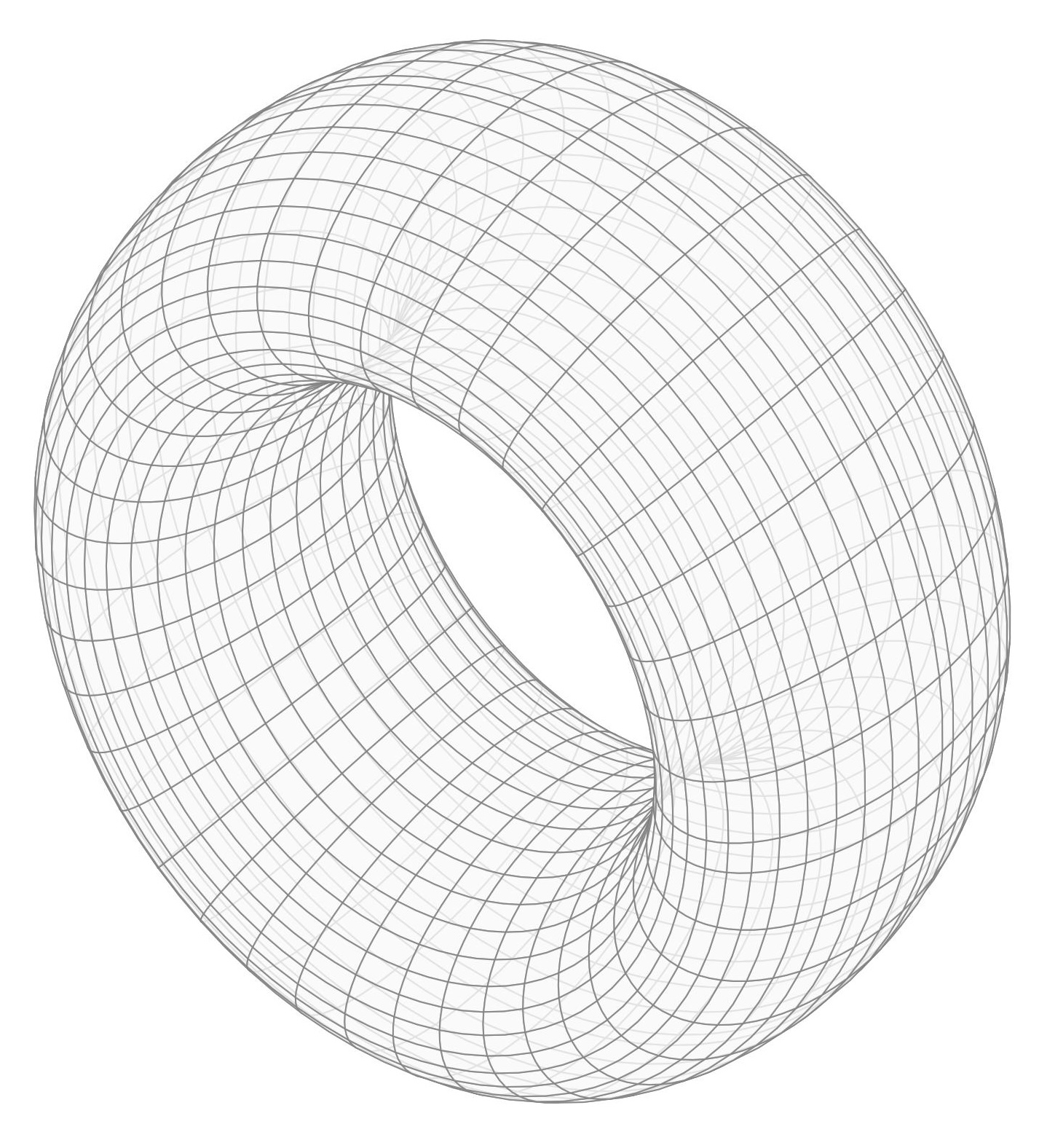}};
            \begin{scope}[x={(img.south east)}, y={(img.north west)}]
                \draw[red, thick, fill] (0.45, 0.92) circle (0.009);
                \node[above] at (0.45, 0.92) {\small $P_2$};

                \draw[red, thick, fill] (0.5, 0.3) circle (0.009);
                \node[above] at (0.5, 0.3) {$R_1$};
            \end{scope}
        \end{tikzpicture}
        \caption{\small The graph of $f$ restricted to a fiber}
        \label{fig:hornedtorus_a}  % Changed to unique label
    \end{subfigure}
    \hspace{2cm}
    \begin{subfigure}[b]{0.4\textwidth}
        \centering
        \begin{tikzpicture}
            \node[anchor=south west, inner sep=0] (img)
                {\includegraphics[width=\linewidth]{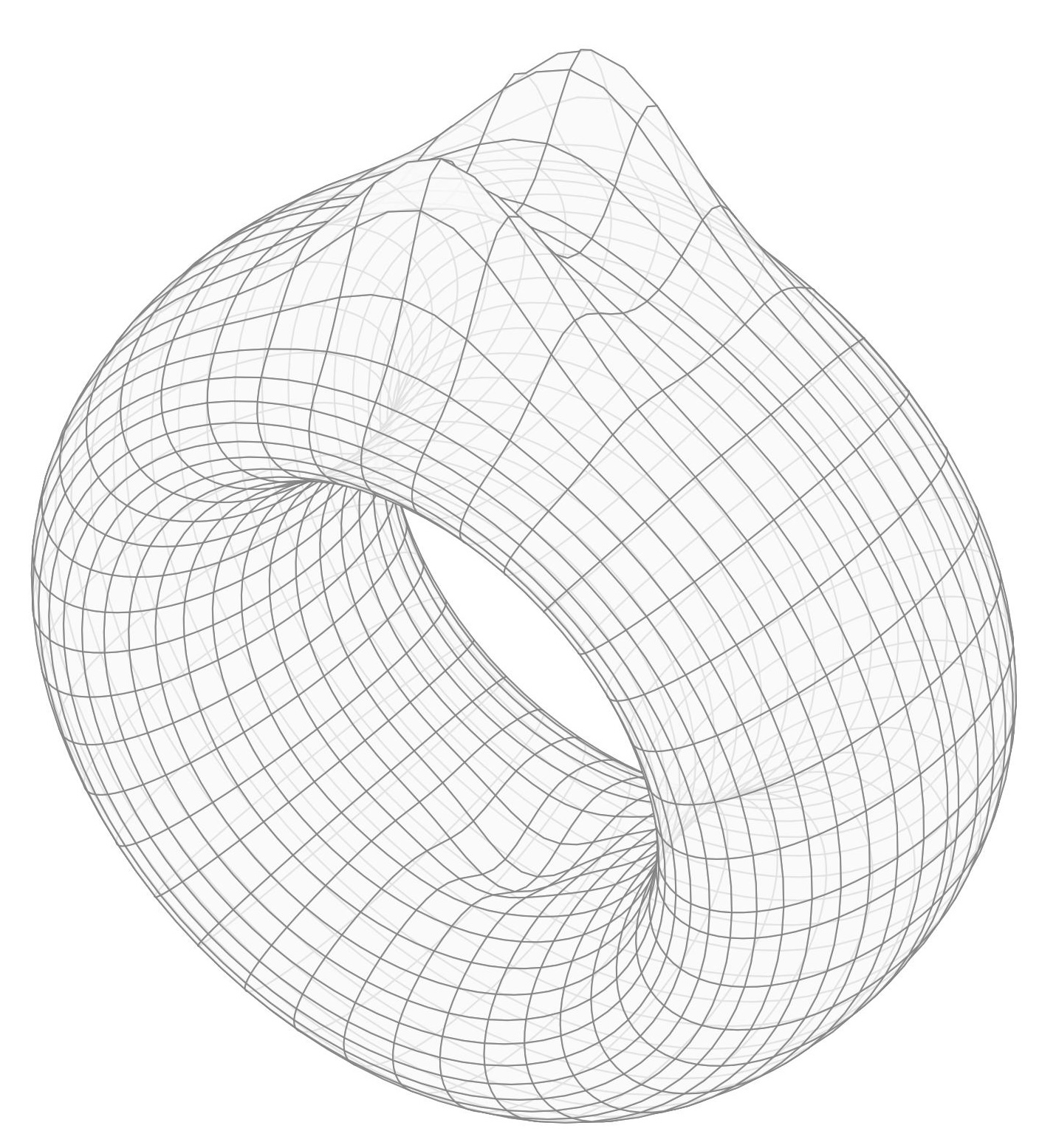}};
            \begin{scope}[x={(img.south east)}, y={(img.north west)}]
                \draw[blue, thick, fill] (0.41, 0.85) rectangle (0.43, 0.87);
                \node[above] at (0.41, 0.85) {\small $P_2'$};
                
                \draw[blue, thick, fill] (0.55, 0.95) rectangle (0.57, 0.97);
                % \node[above] at (0.55, 0.95) {\small $S_2'$};

                \draw[red, thick, fill] (0.5, 0.83) circle (0.009);
                \node[right] at (0.5, 0.83) {\small $\bar{P}_1$};

                \draw[red, thick, fill] (0.5, 0.23) circle (0.009);
                \node[right] at (0.5, 0.23) {\small $\bar{R}_{0}$};

                \draw[purple, thick, fill] (0.43, 0.21) rectangle (0.45, 0.23);
                \node[below] at (0.43, 0.21) {\small $R_{1}'$};

                \draw[purple, thick, fill] (0.58, 0.32) rectangle (0.6, 0.34);
                % \node[left] at (0.58, 0.32) {\small $S_{1,2}'$};
            \end{scope}
        \end{tikzpicture}
        \caption{\small The graph of $f'$ restricted to a fiber}
        \label{fig:hornedtorus_b}  % Changed to unique label
    \end{subfigure}
    \caption{Effects of stabilization on a torus fiber}  % Moved caption before label
    \label{fig:deformed_shapes2}
\end{figure}
\end{example}

\bibliographystyle{alpha}
\bibliography{mybib}

\end{document}